\documentclass[10pt,reqno]{amsart}

\usepackage{graphicx}
\graphicspath{ {./images/} }
\usepackage{xargs}
\usepackage{geometry}
\usepackage{fancyhdr} 
\usepackage{lastpage} 
\usepackage{extramarks} 
\usepackage[most]{tcolorbox} 
\usepackage{xcolor} 
\usepackage[hidelinks]{hyperref} 
\usepackage[permil]{overpic}
\usepackage{pict2e} 
\usepackage{listings}
\usepackage{amssymb}
\usepackage{amsthm}
\theoremstyle{plain}
\usepackage[
backend=biber,
style=numeric,
sorting=ynt
]{biblatex}
\usepackage{xargs}
\usepackage{float}
\usepackage{lipsum}
\usepackage{geometry}
\usepackage{fancyhdr} 
\usepackage{lastpage} 
\usepackage{extramarks} 
\usepackage[most]{tcolorbox} 

\usepackage{xcolor} 
\usepackage[hidelinks]{hyperref} 
\usepackage[permil]{overpic}
\usepackage{pict2e} 
\usepackage{listings}

\newtheorem*{theorem*}{Theorem}

\theoremstyle{notation}

\numberwithin{equation}{section}
\theoremstyle{plain}
\newtheorem{definition}{Definition}[section]
\newtheorem{theorem}[equation]{Theorem}
\newtheorem{claim}[equation]{Theorem}
\newtheorem{remark}[equation]{Remark}
\newtheorem{corollary}[equation]{Corollary}
\newtheorem{lemma}[equation]{Lemma}
\newtheorem{conj}[equation]{Conjecture}
\newtheorem{proposition}[equation]{Proposition} 
\providecommand{\keywords}[1]
{
  \small	
  \textbf{\textit{Keywords---}} #1
}

\title{Topological lower bounds for the Rössler System}
\author{Eran Igra}

\begin{document}

\begin{abstract}
The Rössler System is one of the best known chaotic dynamical systems, exhibiting a plethora of complex phenomena - and yet, only a few studies tackled its complexity analytically. Building on previous work by the author, in this paper we characterize the dynamical complexity for the Rössler System at parameter values at which the flow satisfies a certain heteroclinic condition. This will allow us to characterize the knot type of infinitely many periodic trajectories for the flow - and reduce the Rössler system to a simpler hyperbolic flow, capturing its essential dynamics.
\end{abstract}

\maketitle
\section{Introduction}

In 1976, Otto E. Rössler introduced the following system of Ordinary Differential Equations, depending on parameters $A,B,C\in\mathbf{R}^{3}$ \cite{Ross76}:

\begin{equation} \label{Vect}
\begin{cases}
\dot{X} = -Y-Z \\
 \dot{Y} = X+AY\\
 \dot{Z}=B+Z(X-C)
\end{cases}
\end{equation}

Inspired by the Lorenz attractor (see \cite{Lo}), Otto E. Rössler attempted to find the simplest non-linear flow exhibiting chaotic dynamics. This is realized by the vector field above, as it has precisely one non-linearity, $XZ$ - and as observed by Rössler, the vector field appears to generate a chaotic attractor for $(A,B,C)=(0.2,0.2,5.7)$. In more detail, at these parameter values Rössler observed the first return map of the flow has the shape of a horseshoe (i.e., numerically), which is known to be chaotic (see \cite{S}).\\

\begin{figure}[h]
\centering
\begin{overpic}[width=0.4\textwidth]{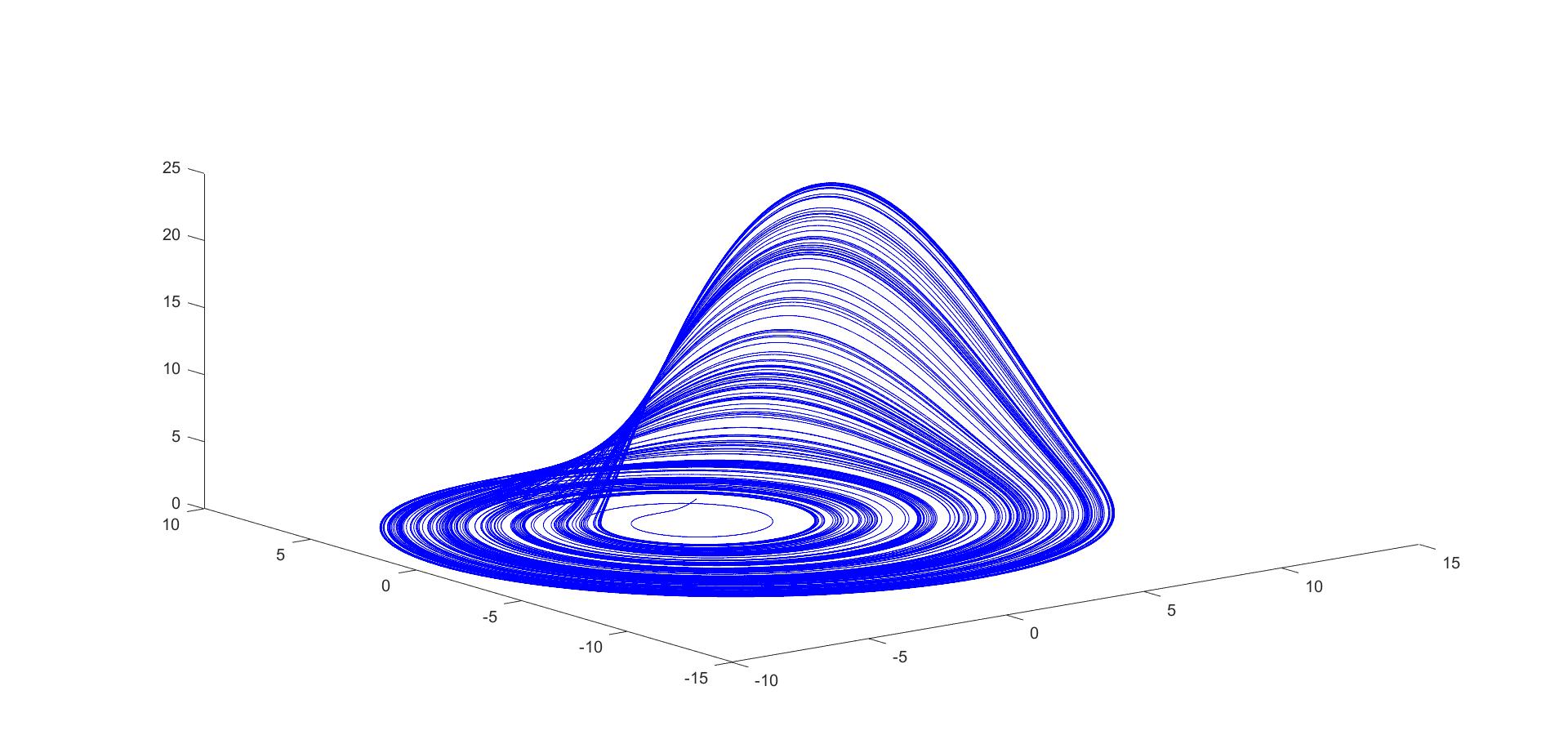}
\end{overpic}
\caption[Fig1]{The Rössler attractor at $(A,B,C)=(0.2,0.2,5.7)$}
\end{figure}

Since its introduction in 1976, the Rössler system was the focus of many numerical studies - despite the simplicity of the vector field, the flow gives rise to many non-linear phenomena (see, for example: \cite{MBKPS}, \cite{BBS}, \cite{G}, \cite{BB2}, \cite{SR}). One particular feature is that varying the parameters $A,B,C$ often leads to a change in the collection of knot types of the periodic trajectories on the attractor - see, for example, \cite{Le},\cite{RO}.\\

To our knowledge, no study attempted to analytically explain the topological changes in the attractor described above - and it is precisely this gap that our results address. Motivated by the results of both \cite{I} and \cite{Pi}, in this paper we prove that provided the Rössler system satisfies a certain heteroclinic condition at some specific parameter values $(A,B,C)$ one can reduce the flow to a simpler model - as well as classify the knot-types generated as periodic trajectories using Template Theory (see Th.\ref{BIRW} in the next section). Roughly speaking, we say a parameter $(A,B,C)$ is a \textbf{trefoil parameter} provided the corresponding Rössler system generates a heteroclinic knot whose knot-type is the trefoil knot (see Def.\ref{def32}).\\

With this notion in  mind, we now introduce our results. We begin with the following Theorem, which we prove in Section $2$ (see Th.\ref{deform}):
\begin{claim}\label{th4}
    Let $(A,B,C)=p$ be a trefoil parameter for the Rössler system, and denote by $F_p$ the corresponding vector field w.r.t. Eq.\ref{Vect}. Then, there exists a smooth vector field $G$, hyperbolic on a bounded, invariant set in $\mathbf{R}^3$, s.t. every periodic trajectory for $G$ (save possibly for two) can be deformed by an isotopy of $\mathbf{R}^3$ to a periodic trajectory of $F_p$.
\end{claim}

Now, let $f$ denote a surface homeomorphism $f:S\to S$. Provided $S$ is also of negative Euler characteristic, there exists a dynamically-minimal homeomorphism $g:S\to S$ s.t. $f$ is both isotopic to $g$, and $g$ is semi-conjugate $f$ (see, for example, Th.5.1 and Th.6.2 in \cite{Bo}). Moreover, every periodic orbit of $g$ can be continuously deformed to a periodic orbit of $f$ (see Th.1 and Th.2 in \cite{Han}). Since first-return maps for smooth flows are surface diffeomorphisms, one would expect similar results to hold for three-dimensional flows - and yet, the problem of generalizing such facts to three-dimensional flows is far from solved (see \cite{M} for a survey of this problem). In this context, it is easy to see the vector field $G$ given by Th.\ref{th4} is an analogue of the dynamically minimal maps described above - as such, Th.\ref{th4} can be considered as a step towards solving this problem.\\

Now, recall that given a hyperbolic flow $\xi_t,t\in\mathbf{R}$ defined in $\mathbf{R}^3$ (see Def.\ref{hyperbolicvector}), there exists a branched surface - often referred to as a Template (see Def.\ref{deftemp} in Section $2$) - which encodes every knot-type generated by $\xi_t$ as a periodic trajectory (see Th.\ref{BIRW} in the next section for the precise formulation). Inspired by Th.\ref{th4}, we apply Th.\ref{th4} to prove the following Theorem in Section $3$ (see Th.\ref{geometric}):

\begin{claim}
\label{geomod}    Let $(A,B,C)=p$ be a trefoil parameter for the Rössler system, and denote by $F_p$ the corresponding vector field. Then, there exists a Template, denoted by $L(0,1)$, s.t. any knot type in $L(0,1)$ (save possibly for two) is realized as a periodic trajectory for $F_p$. Consequentially, the  Rössler system at trefoil parameters generates infinitely many periodic trajectories whose knot type is that of a torus knot (see Def.\ref{torusknot}). 
\end{claim}
Th.\ref{geomod} proves the dynamics of $G$ can be thought of as a "topological lower bound" for the complexity of the Rössler system at trefoil parameter - thus making it "essentially hyperbolic". Consequentially, it follows by Th.\ref{geomod} that the dynamics of the Rössler system at trefoil parameter are essentially hyperbolic, regardless of whether the vector field actually splits the tangent bundle to stable and unstable bundles. Consequentially, Th.\ref{geometric} can be thought of as a step towards proving the Chaotic Hypothesis, introduced in \cite{gal} - which states the flow on a given chaotic attractor is essentially hyperbolic.\\

Finally, we conclude this paper with the following result, which answers how much of the complex dynamics at trefoil parameters persist under perturbations. Using Th.\ref{geometric} and the Fixed-Point Index for maps, we prove the following result (see Th.\ref{persistence} in Section $4$):

\begin{claim}
\label{perss}    Let $(A,B,C)=p$ be a trefoil parameter, and assume $T$ is a knot type on encoded by the Template $L(0,1)$, given by Th.\ref{geomod}. Then, given any parameter $v=(A',B',C')$ sufficiently close to $p$, the Rössler system corresponding to $v$ generates a periodic trajectory whose knot type is $T$.
\end{claim}

That is, Th.\ref{perss} has the following meaning - the closer a parameter $v=(A,B,C)$  is to a trefoil parameter, the more complex are the dynamics of the corresponding Rössler system. As such, in many respects, Th.\ref{perss} is a topological version of Th.4.4 in \cite{I} - where a similar result was proven for the first-return map using symbolic dynamics.

\section*{Preliminaries}
In this section we review several of the preliminaries required to prove our results in sections $3,4$ and $5$. This section is organized as follows - we begin with a quick introduction to the Rössler system and state several results proven in \cite{I} - in particular, we rigorously define the notion of a trefoil parameter, and recall its connection to the onset of chaos. Later, we review several facts from Template Theory, and state how it is used to analyze the periodic dynamics of three-dimensional flows. 

\subsection{Trefoil parameters and chaotic dynamics in the Rössler system.}

From now on given $(a,b,c)\in\mathbf{R}^3$, we switch to the more convenient form of the Rössler system:
\begin{equation} \label{Field}
\begin{cases}
\dot{x} = -y-z \\
 \dot{y} = x+ay\\
 \dot{z}=bx+z(x-c)
\end{cases}
\end{equation}
We always denote this vector field corresponding to $(a,b,c)\in\mathbf{R}^3$ by $F_{a,b,c}$. This definition is slightly different from the one presented in Eq.\ref{Vect} - however, setting $p_1=\frac{-C+\sqrt{C^2-4AB}}{2A}$, it is easy to see that whenever $C^2-4AB>0$, $(X,Y,Z)=(x-ap_1,y+p_1,z-p_1)$ defines a change of coordinates between the vector fields in Eq.\ref{Vect} and Eq.\ref{Field}.\\

Since the vector field in Eq.\ref{Field} depends on three parameters, $(a,b,c)$, we now specify the region in the parameter space in which we prove our results. The parameter space $P\subseteq\mathbf{R}^3$ we consider throughout this paper is composed of parameters satisfying the following\label{eq:9}:

\begin{itemize}
\item \textbf{Assumption $1$} -  for every parameter $p\in P,p=(a,b,c)$ the parameters satisfy $a,b\in(0,1)$ and $c>1$. For every choice of such $p$, the vector field $F_p$ given by Eq.\ref{Field} always generates precisely two fixed points - $P_{In}=(0,0,0)$ and $P_{Out}=(c-ab,b-\frac{c}{a},\frac{c}{a}-b)$.
\item \textbf{Assumption 2 }- for every $p\in P$ the fixed points $P_{In},P_{Out}$ are both saddle-foci of opposing indices. In more detail, we always assume that $P_{In}$ has a one-dimensional stable manifold, $W^s_{In}$, and a two-dimensional unstable manifold, $W^u_{In}$. Conversely, we always assume $P_{Out}$ has a one-dimensional unstable manifold, $W^u_{Out}$, and a two-dimensional stable manifold, $W^s_{Out}$ (see the illustration in Fig.\ref{loci}). 
\item \textbf{Assumption 3 }- For every $p\in P$, let $\gamma_{In}<0$ and $\rho_{In}\pm i\psi_{In}$, $\rho_{In}>0$ denote the eigenvalues of $J_p(P_{In})$, the linearization of $F_p$ at $P_{In}$, and set $\nu_{In}=|\frac{\rho_{In}}{\gamma_{In}}|$. Conversely, let $\gamma_{Out}>0$, $\rho_{Out}\pm i\psi_{Out}$ s.t. $\rho_{Out}<0$ denote the eigenvalues of $J_p(P_{Out})$, the linearization at $P_{Out}$, and define $\nu_{Out}=|\frac{\rho_{Out}}{\gamma_{Out}}|$. We will refer to $\nu_{In},\nu_{Out}$ as the respective saddle indices at $P_{In},P_{Out}$, and we will always assume $(\nu_{In}<1)\lor(\nu_{Out}<1)$ - that is, for every $p\in P$ at least one of the fixed points satisfies the \textbf{Shilnikov condition} (see Th.13.8 in \cite{SSTC} or \cite{LeS} for more details on the connection between the Shilnikov condition and the onset of chaos). 
\end{itemize}

\begin{figure}[h]
\centering
\begin{overpic}[width=0.4\textwidth]{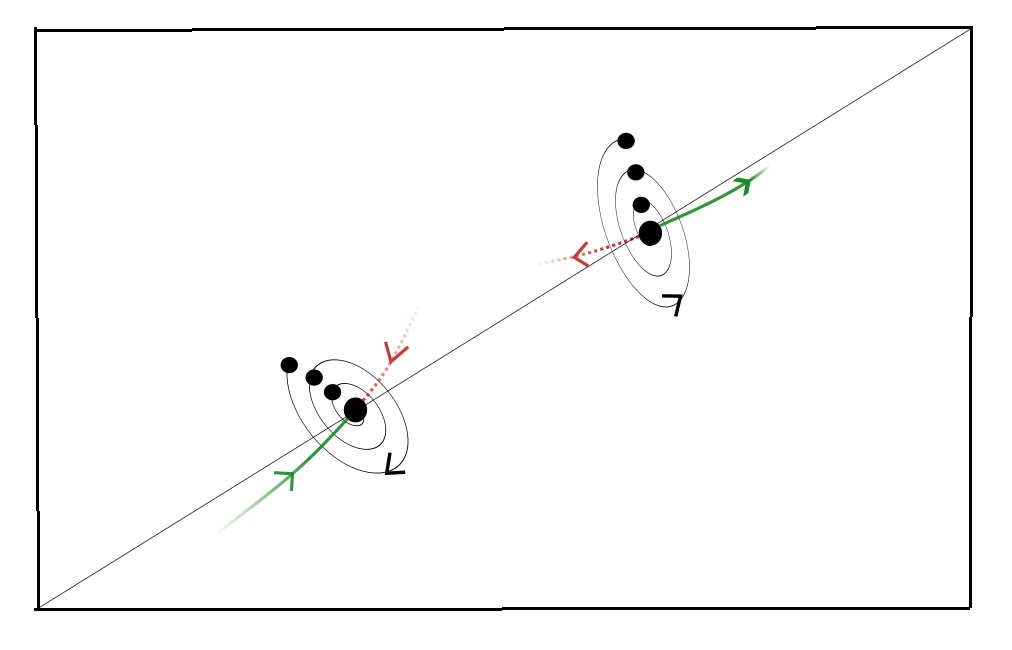}%
\put(630,280){$W^s_{Out}$}
\put(750,430){$W^u_{Out}$}
\put(690,390){$P_{Out}$}
\put(390,130){$W^u_{In}$}
\put(410,220){$P_{In}$}
\put(440,420){$U_p$}
\put(530,280){$l_p$}
\put(570,180){$L_p$}
\put(250,110){$W^s_{In}$}
\end{overpic}
\caption[Fig2]{The local dynamics around the fixed points. The green and red flow lines are the one-dimensional components in $W^s_{In},W^u_{Out}$. $U_p$ is the half-plane above the line $l_p$.}\label{loci}
\end{figure}
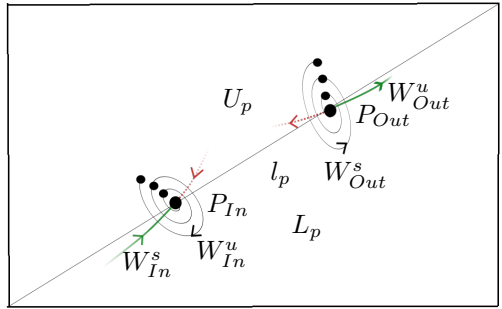

It is easy to see the parameter space $P$ we are considering is an open set in $\mathbf{R}^3$. In addition, it includes the parameter space for the Rössler system considered in \cite{BBS},\cite{MBKPS} and \cite{G}, where many interesting bifurcation phenomena were observed. As far as the general topological dynamics of vector field $F_p$, $p\in P$ are concerned, we have the following result, proven in \cite{I} using Th.1 from \cite{LiLl} (see Th.2.8, Lemma 2.6 and Cor.2.5  in \cite{I}):

\begin{claim}
  \label{th21}  For every parameter $p\in P$, the vector field $F_p$ can be continuously extended to a vector field of $S^3$ with three fixed points: the saddle-foci $P_{In}$ and $P_{Out}$, and a degenerate fixed point at $\infty$ of index $0$ - moreover, $F_v$ is smooth away from $\infty$. Consequentially, the Rössler system corresponding to $p$ generates two heteroclinic trajectories:
\begin{itemize}
    \item $\Gamma_{In}\subseteq W^s_{In}$, which connects $P_{In},\infty$ in $S^3$. Moreover, $\Gamma_{In}$ is trapped in the region $\{\dot{y}\geq0\}\cap\{(x,y,z)|y\geq0\}$ (were $\dot{y}$ is taken w.r.t. Eq.\ref{Field}).
    \item $\Gamma_{Out}\subseteq W^u_{Out}$, which connects $P_{Out}$, $\infty$ in $S^3$. Similarly, $\Gamma_{Out}$ is trapped in the region $\{\dot{y}\geq0\}\cap\{(x,y,z)|\leq\frac{ab-c}{a}\}$.
\end{itemize}
As a consequence, for every sufficiently large $r>0$, there exists a smooth vector vector field on $S^3$, $R_p$, s.t.:
\begin{itemize}
    \item $R_p$ coincides with $F_p$ on the open ball $B_r(P_{In})$.
    \item $R_p$ has precisely two fixed points in $S^3$ - namely, the saddle foci $P_{In}$ and $P_{Out}$.
    \item $R_p$ generates a heteroclinic trajectory which connects $P_{In},P_{Out}$ and passes through $\infty$.
\end{itemize}
\end{claim}

\begin{figure}[h]
\centering
\begin{overpic}[width=0.4\textwidth]{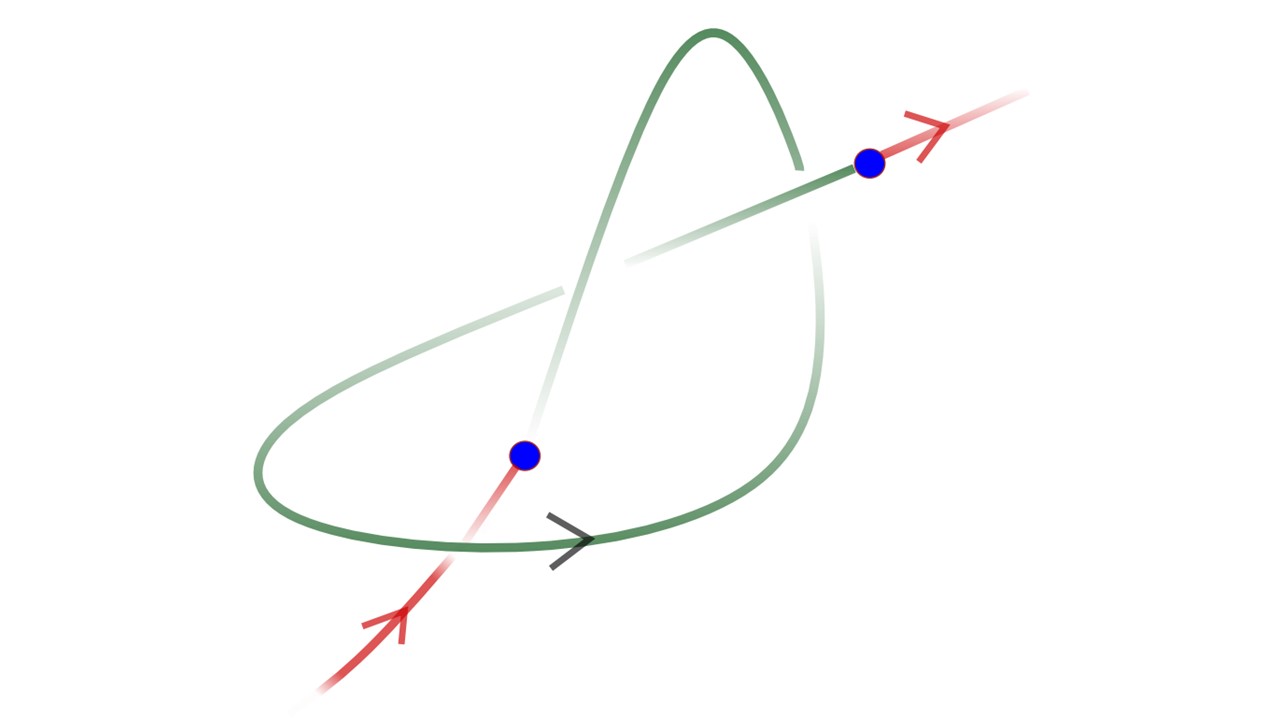}
\put(440,210){$P_{In}$}
\put(340,70){$\Gamma_{In}$}
\put(400,350){$\Theta$}
\put(660,370){$P_{Out}$}
\put(780,430){$\Gamma_{Out}$}
\end{overpic}
\caption[Fig7]{A heteroclinic trefoil knot (see Def.\ref{def32}). $\Theta$ denotes the bounded heteroclinic trajectory, while $\Gamma_{In},\Gamma_{Out}$ denote the unbounded heteroclinic trajectories given by Th.\ref{th21}.}
\label{fig7}
\end{figure}
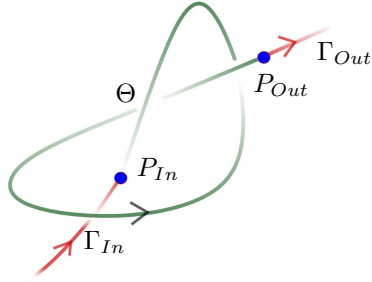
For an illustration, see Fig.\ref{fig7} and \ref{TRANS}. We will also need several other general results from \cite{I}. To recall these results, choose some $p\in P$, $p=(a,b,c)$, and consider the cross-section $Y=\{\dot{y}=0\}=\{(x,-\frac{x}{a},z)|x,z\in\mathbf{R}\}$ defined by Eq.\ref{Field}, and set $l_p=\{(x,-\frac{x}{a},\frac{x}{a})|x\in\mathbf{R}\}$ (see the illustration in Fig.\ref{loci}). Because the normal vector to $Y$ is $N=(1,a,0)$, it follows by direct computation that $l_p=\{s\in Y|F_v(s)\bullet N=0\}$, and $P_{In},P_{Out}\in l_p$. It follows $Y\setminus l_p$ constitutes of two components, both half planes, parameterized as follows:

\begin{itemize}
    \item $U_p=\{(x,-\frac{x}{a},z)|x\in\mathbf{R}, \frac{x}{a}<z\}$ - that is, the upper half plane (see the illustration in Fig.\ref{loci}). On $U_p$ the vector field $F_p$ points into the half-space $\{\dot{y}<0\}$.
    \item $L_p=\{(x,-\frac{x}{a},z)|x\in\mathbf{R}, \frac{x}{a}>z\}$ - that is, the lower half plane (see the illustration in Fig.\ref{loci}). On $L_p$ the vector field $F_p$ points into the half-space $\{\dot{y}>0\}$.
\end{itemize}

It is easy to see both cross-sections $U_p$ and $L_p$ vary smoothly when we vary smoothly the parameter $p$ in the parameter space. As proven in Lemma 2.2 and Lemma 2.1 in \cite{I}, the local dynamics on $U_p$ satisfy the following:

\begin{corollary}
\label{TR}    For any parameter $p\in P$, the cross-section $U_p$ satisfies the following:
    \begin{itemize}
        \item The two-dimensional $W^u_{In},W^s_{Out}$ are transverse to $U_p$ at $P_{In},P_{Out}$ (respectively - see the illustration in Fig.\ref{loci})).
        \item Given any $s\in\mathbf{R}^3$, denote its trajectory by $\gamma_s(t)$, $t\in\mathbf{R},\gamma_s(0)=0$. Provided $\Gamma(s)=\{\gamma_s(t)|t>0\}$ is bounded and $\Gamma(s)$ does not limit to a fixed point, there exists some $t_0>0$ s.t. $\gamma_s(t_0)$ is a transverse intersection point between $U_p,\Gamma(s)$. Consequentially, every periodic trajectory hits $U_p$ transversely at least once.
    \end{itemize}
\end{corollary}

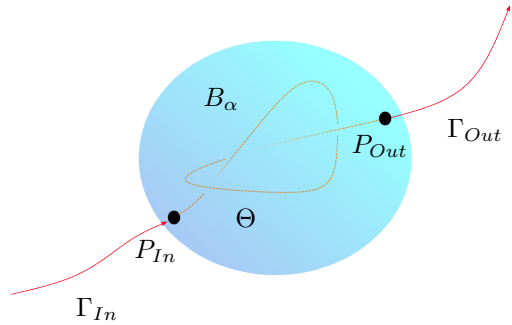
\begin{figure}[h]
\centering
\begin{overpic}[width=0.4\textwidth]{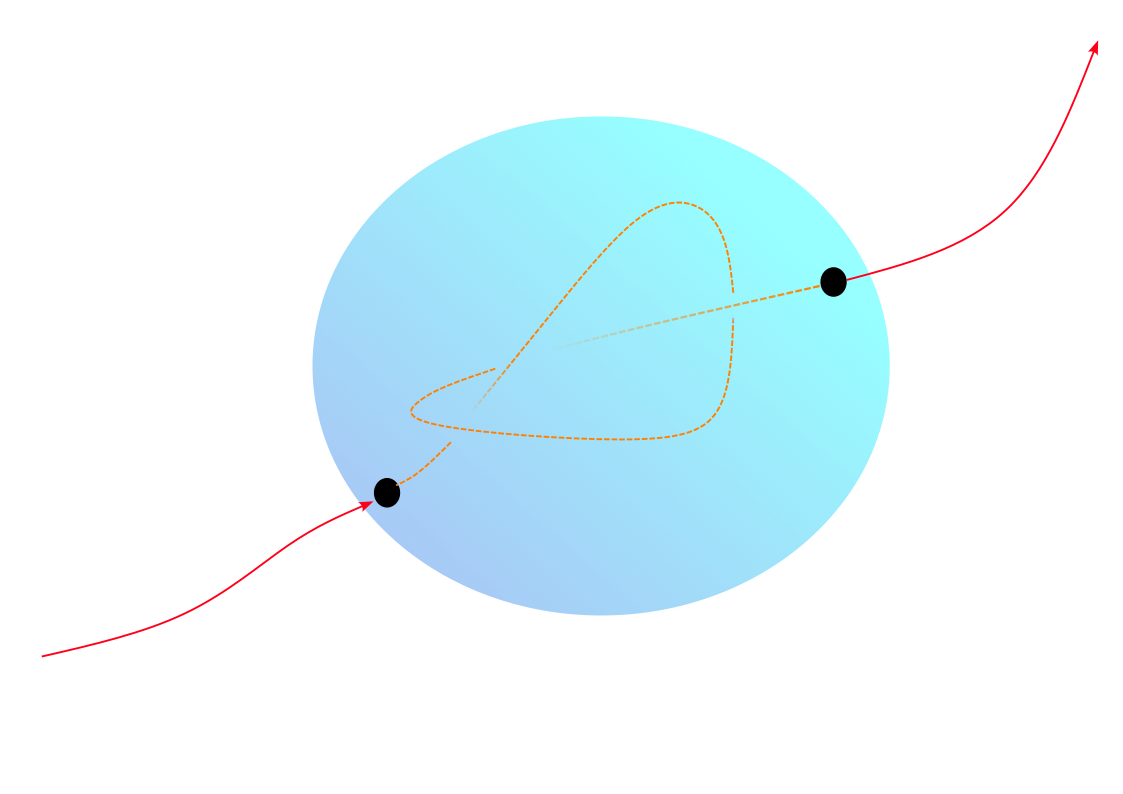}
\put(460,250){$\Theta$}
\put(270,190){$P_{In}$}
\put(160,80){$\Gamma_{In}$}
\put(400,480){$B_\alpha$}
\put(680,390){$P_{Out}$}
\put(860,420){$\Gamma_{Out}$}
\end{overpic}
\caption[The ball $B_\alpha$.]{\textit{The heteroclinic trajectory $\Theta$ trapped inside the topological ball $B_\alpha$, sketched as a the blue sphere.}}
\label{boundedtref}
\end{figure}

As stated in the introduction, in \cite{I} the author had proven a criterion for the existence of complex dynamics for the Rössler system. Since the proofs of both Th.\ref{geometric} and \ref{deform} are heavily based on that criterion, let us now introduce it. In order to do so, we first make the following observation - assume $p\in P$ is a parameter s.t. the vector field $F_p$ generates a bounded, heteroclinic trajectory $\Theta$ which flows from $P_{Out}$ to $P_{In}$ - in particular, $\Theta=W^s_{In}\cap W^u_{Out}$ (see the illustration in Fig.\ref{fig7}). Now, consider the set $\Lambda=W^s_{In}\cup W^u_{Out}\cup\{P_{In},P_{Out},\infty\}$. Since $W^s_{In}=\Theta\cup\Gamma_{In}$ and $W^u_{Out}=\Theta\cup\Gamma_{Out}$ by Th.\ref{th21} it immediately follows $\Lambda$ is a knot in $\mathbf{R}^3$. Motivated by this observation, from now on we consider a very specific type of such heteroclinic knots, trefoil parameters, defined below:
\begin{definition}\label{def32}
With the notations above, we say $p=(a,b,c)\in P$ is a \textbf{trefoil parameter for the Rössler system} provided the following three conditions are satisfied by the corresponding vector field $F_p$:
\begin{itemize}
    \item There exists a bounded heteroclinic trajectory $\Theta$ as in Fig.\ref{fig7}. Consequentially, $\Lambda$ (as defined above) forms a trefoil knot in $S^3$.
    \item The two-dimensional manifolds $W^u_{In}$ and $W^s_{Out}$ coincide. This condition implies $\overline{W^u_{In}}=\overline{W^s_{Out}}$ forms the boundary of an open topological ball - which, from now on, we always denote by $B_\alpha$. It is easy to see $\Theta\subseteq B_\alpha$, while $\Gamma_{In},\Gamma_{Out}\not\subseteq B_\alpha$ (see the illustration in Fig.\ref{boundedtref}).
    \item $\Theta\cap \overline{U_p}=\{P_0\}$ is a point of transverse intersection - see the illustration at Fig.\ref{intersect} and Fig.\ref{fig9}.
    \item  $\Theta\cap \overline{U_p}=\{P_0\}$ is a point of transverse intersection - see the illustration at Fig.\ref{intersect} and Fig.\ref{fig9}\end{itemize}.
\end{definition}

\begin{figure}[h]
\centering
\begin{overpic}[width=0.4\textwidth]{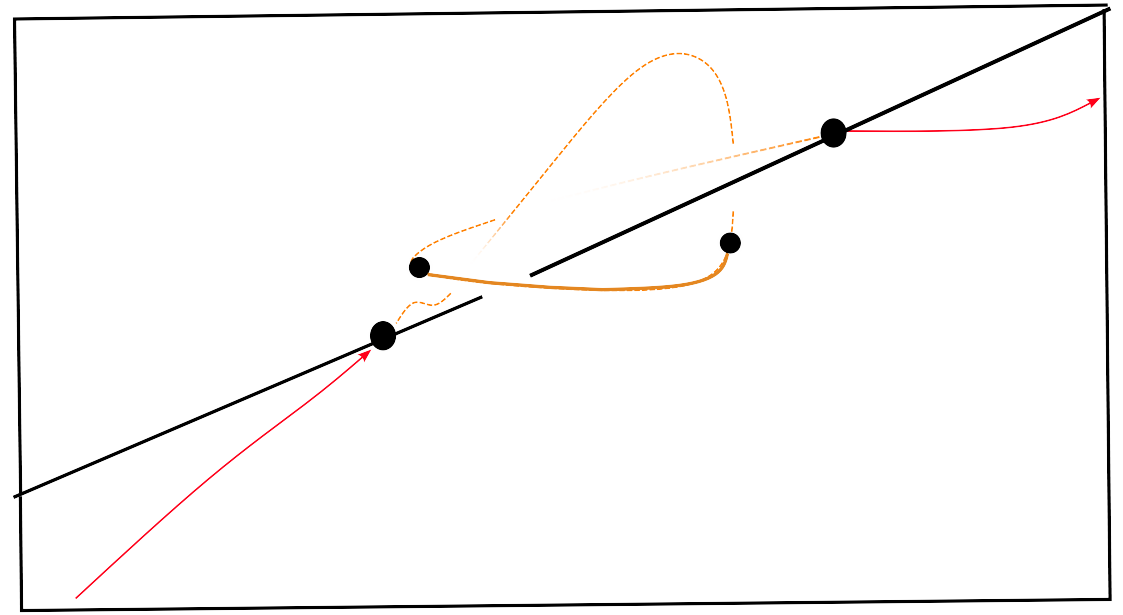}
\put(480,250){$\Theta$}
\put(320,180){$P_{In}$}
\put(160,50){$\Gamma_{In}$}
\put(310,340){$P_0$}
\put(670,310){$P_1$}
\put(100,450){$U_p$}
\put(700,50){$L_p$}
\put(730,370){$P_{Out}$}
\put(900,400){$\Gamma_{Out}$}
\end{overpic}
\caption[The intersection of the heteroclinic trefoil with $\{\dot{y}=0\}$.]{\textit{The heteroclinic trajectory $\Theta$ (for a trefoil parameter) winds once around $P_{In}$ - hence it intersects the half-plane $U_p$ at $P_0$ and the half-plane $L_p$ at $P_1$.}}
\label{intersect}
\end{figure}
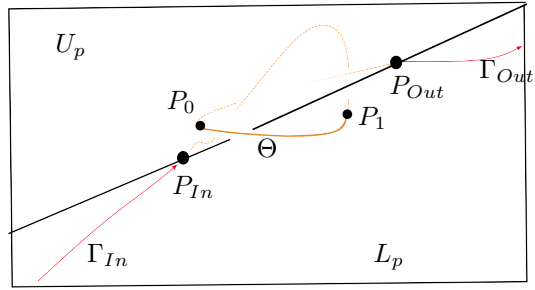

As observed numerically, for a large class of parameters $v\in P$, the two-dimensional invariant manifold shields trajectories from escaping to $\infty$, thus repelling them towards the attractor. It is easy to see something similar happens at trefoil parameters - namely, the trajectories of initial conditions $s\in B_\alpha$ tend towards $P_{Out}$ along $\partial B_\alpha=W^s_{Out}$ until getting repelled along the bounded heteroclinic trajectory $\Theta$ towards $P_{In}$. Additionally, the existence of parameters $p\in P$ at which a heteroclinic trajectory $\Theta$ as in Fig.\ref{fig7} was observed numerically (see Fig.$5.B1$ in \cite{MBKPS}). In light of this discussion, it is easy to see that trefoil parameters can be considered as an idealized form of the Rössler system (for a more detailed discussion, see the discussion at the beginning of Section $3$ in \cite{I}).\\ 

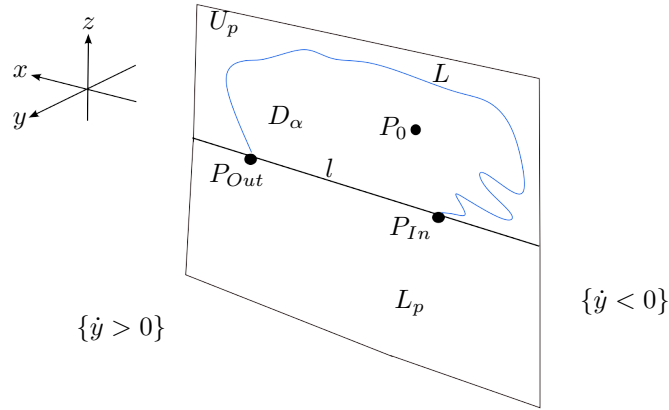
\begin{figure}[h]
\centering
\begin{overpic}[width=0.4\textwidth]{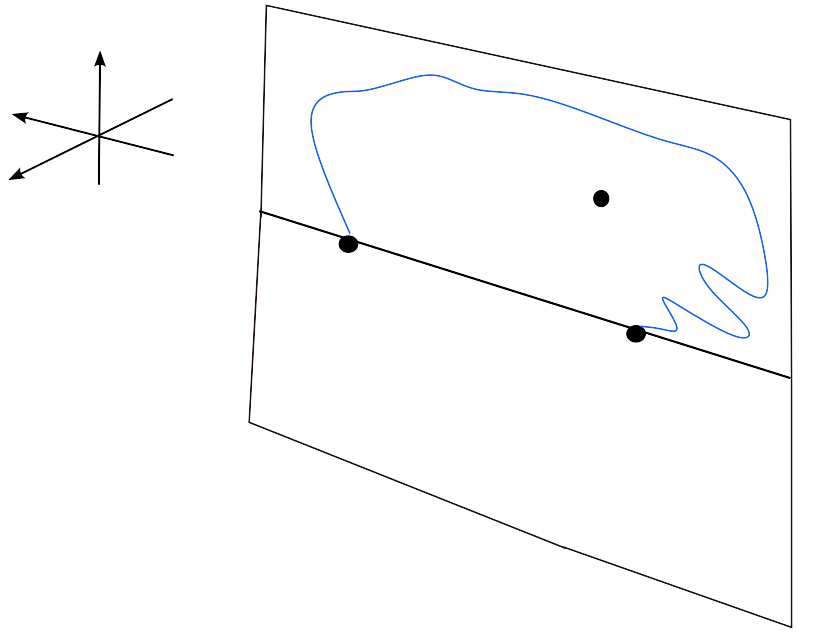}
\put(690,340){$P_{In}$}
\put(670,530){$P_0$}
\put(1050,200){$\{\dot{y}<0\}$}
\put(100,150){$\{\dot{y}>0\}$}
\put(350,440){$P_{Out}$}
\put(570,450){$l$}
\put(770,630){$L$}
\put(350,730){$U_p$}
\put(460,550){$D_\alpha$}
\put(700,200){$L_p$}
\put(-20,640){$x$}
\put(-20,550){$y$}
\put(110,730){$z$}
\end{overpic}
\caption[The disc $D_\alpha$.]{\textit{The Jordan domain $D_\alpha\subseteq U_p$, bounded by the blue curve $L=\partial B_\alpha\cap U_p$, the fixed points, and the boundary arc $l$.}}
\label{fig9}
\end{figure}

The reason we are interested in trefoil parameters is because, as stated earlier, at trefoil parameters one can prove the flow dynamics are complex - i.e., at trefoil parameters the Rössler system generates an infinite collection of periodic trajectories. In order to state these results, we first recall several facts from \cite{I} about the first-return map for the flow at trefoil parameters. We begin with the following fact, which is the amalgamation of Lemmas 3.1, 3.2, Prop.3.3 and Prop.3.4 proven in \cite{I}:

\begin{proposition}
  \label{boundball} Let $p\in P$ be a trefoil parameter, and set $D_\alpha=B_\alpha\cap U_p$, $l=B_\alpha\cap l_p$ (where $U_p$ and $l_p$ are as defined immediately before Cor.\ref{TR}). Then, we have the following:

  \begin{itemize}
      \item The set $L=\partial B_\alpha\cap U_p$ is a curve, connecting $P_{In}$ and $P_{Out}$, and $l$ is a straight line connecting $P_{In}$ and $P_{Out}$. Consequentially, $D_\alpha$ is a topological disc on $U_p$ (see the illustration in Fig.\ref{fig9}). 
      \item Consequentially, $\partial D_\alpha$ in $U_p$ consists of the union $L\cup l\cup\{P_{In},P_{Out}\}$ (see the illustration in Fig.\ref{fig9}).
      \item $P_0$, the intersection of the bounded heteroclinic trajectory $\Theta$ with $U_p$, is interior to $D_\alpha$ (see the illustration in Fig.\ref{fig9}).
      \item The first-return map $f_p:\overline{D_\alpha}\setminus\{P_0\}\to\overline{D_\alpha}\setminus\{P_0\}$ is well-defined on the cross-section $\overline{D_\alpha}\setminus\{P_0\}$. In particular, $\overline{D_\alpha}\setminus\{P_0\}$ is homeomorphic to a punctured closed disc (see the illustration in Fig.\ref{fig9}).
  \end{itemize}
\end{proposition}

The first-return map $f_p:\overline{D_\alpha}\setminus\{P_0\}\to\overline{D_\alpha}\setminus\{P_0\}$ is not continuous on the punctured disc $D_\alpha\setminus\{P_0\}$ - in fact, as far as its continuity properties are concerned we have the following fact (see Prop.3.8, Lemma 3.12 and Cor.3.13 in \cite{I}):

\begin{proposition}
    \label{lem33}
    Let $p\in P$ be a trefoil parameter, and let $l$, $L$ and $f_p:\overline{D_\alpha}\setminus\{P_0\}\to\overline{D_\alpha}\setminus\{P_0\}$ be the first-return map as in Th.\ref{boundball}. Then, we have the following:
    \begin{itemize}
        \item $f_p(L)=L$, and $f_p$ is continuous on $L$.
        \item There exists a component $\delta$ of $f^{-1}_p(l)\cap D_\alpha$ s.t. its closure $\overline\delta$ is homeomorphic to a closed interval, and connects $P_0$ and some $\delta_0\in l$ (see the illustration in Fig.\ref{fig91}). In particular, $\delta_0\ne P_{In}$.
        \item There exist a topological disc $H_p\subseteq D_\alpha$ s.t. $P_{In},P_{Out}$ and $L$ are all subsets of $\partial H_p$ - while $\delta$ and $P_0$ are at most subsets of $\partial H_p$ (where the boundary is taken in $\overline{U_p}$ - see the illustration in Fig.\ref{fig91}). 
        \item The first-return map $f_p:\overline{H_p}\setminus\overline{\delta}\to\overline{D_\alpha}$ is continuous. Moreover, $P_0\not\in f_p(\overline{H_p})$ while $P_{In},P_{Out}$ and $L$ are all subsets of $f_p(\overline{H_p})$ (see the illustration in Fig.\ref{firstpic}).  
        \item Finally, given any curve $\gamma\subseteq{D_\alpha}$ whose closure connects $P_{In}$ and $P_0$, $f_p(\gamma\cap H_p)$ includes at least one arc whose closure connects $P_{In}$ and $\overline{l}$.
    \end{itemize}
\end{proposition}
\begin{figure}[h]
\centering
\begin{overpic}[width=0.4\textwidth]{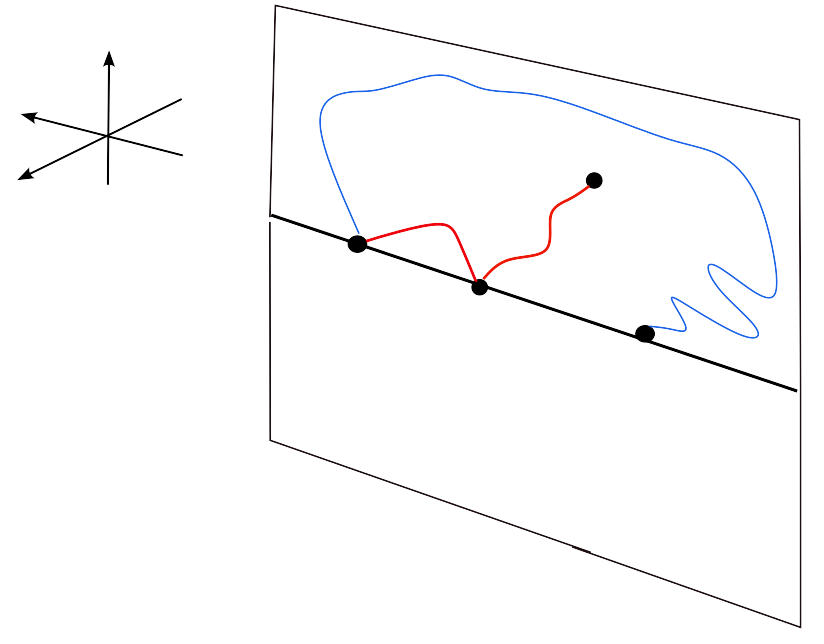}
\put(750,320){$P_{In}$}
\put(380,440){$P_{Out}$}
\put(630,500){$\delta$}
\put(580,370){$\delta_0$}
\put(800,630){$L$}
\put(350,720){$U_p$}
\put(460,620){$H_p$}
\put(650,580){$P_0$}
\put(350,265){$L_p$}
\put(-5,640){$x$}
\put(-5,535){$y$}
\put(125,735){$z$}
\end{overpic}
\caption[Case $B$.]{\textit{The cross-section $H_p$ - in this scenario, $f^{-1}_p(l)$ sketched as the red curve.}}
\label{fig91}
\end{figure}
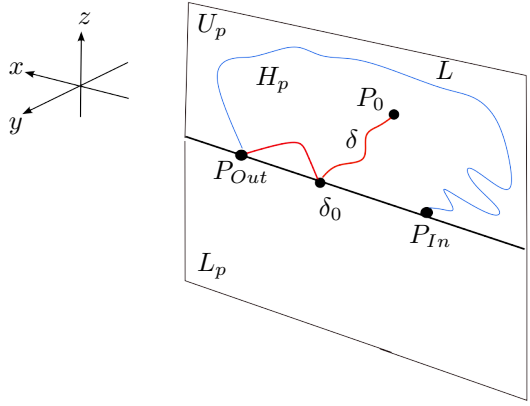
Using the results listed in Prop.\ref{boundball} and Prop.\ref{lem33}, it was proven in \cite{I} the dynamics of the Rössler system at trefoil parameters are chaotic in the following sense: there exists a countable collection of periodic trajectories for the flow inside the topological ball $B_\alpha$. In more detail, the following result was proven: 
\begin{claim}
    \label{th31}
    Let $p\in P$ be a trefoil parameter for the Rössler system, and denote by $\sigma:\{1,2\}^\mathbf{N}\to\{1,2\}^\mathbf{N}$ the one-sided shift. Then, there exists a bounded, infinite, collection of periodic trajectories for the flow inside the topological ball $B_\alpha$. In more detail, consider the first-return map $f_p:\overline{H_p}\setminus\overline{\delta}\to\overline{D_\alpha}$ given by Prop.\ref{lem33} - then, there exists an $f_p$-invariant set $Q\subseteq\overline{H_p}\setminus\delta$ s.t. the following is satisfied:
    \begin{itemize}
        \item $f_p$ is continuous on on $Q$.
        \item There exists a continuous $\pi:Q\to\{1,2\}^\mathbf{N}$ s.t. $\pi\circ f_p=\sigma\circ\pi$. Moreover, $\pi(Q)$ includes every periodic sequence in $\{1,2\}^\mathbf{N}$.
        \item Given any periodic $s\in\{1,2\}^\mathbf{N}$ of minimal period $k$ s.t. $s$ is not the constant $\{1,1,1...\}$, $\pi^{-1}(s)=D_s$ is nonempty and connected. Moreover, $D_s$ includes $x_s$, a periodic point for $f_p$ of minimal period $k$.
    \end{itemize}
\end{claim}

\begin{figure}[h]
\centering
\begin{overpic}[width=0.3\textwidth]{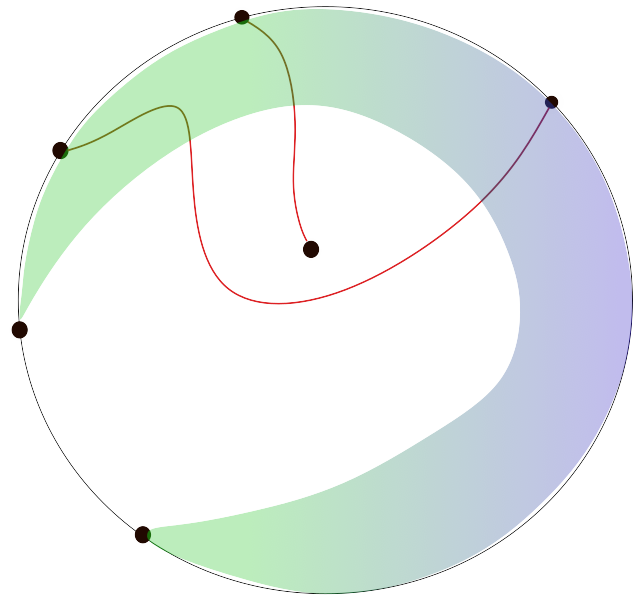}
\put(200,-10){$P_{In}$}
\put(520,550){$P_{0}$}
\put(400,350){$H_p$}
\put(300,430){$\rho$}
\put(600,350){$f_p(l_1)$}
\put(0,750){$r_0$}
\put(-40,390){$r_1$}
\put(30,250){$l$}
\put(350,820){$\delta_0$}
\put(590,850){$f_p(\rho)$}
\put(870,760){$P_{Out}$}
\put(990,380){$L=f_p(L)$}
\end{overpic}
\caption[The first-return map in Case $A$.]{\textit{The cross-section $D_\alpha$ sketched as a disc, where $H_p$ is the region trapped between $\rho,L$ and the arc on $l$ connecting $r_0,P_{In}$. $f_p(H_p)$ imposed on $H_p$ as the crescent-like figure - as can be seen, in this scenario $\overline\delta\cap\overline{H_p}=\emptyset$. $l_1$ is the arc connecting $P_{In}$ and $r_0$ in $\partial H_p$, while $f_p(l_1)$ is the arc connecting $P_{In}$ and $r_1=f_p(r_0)$ (note $f_p(L)=L$).}}
\label{firstpic}
\end{figure}
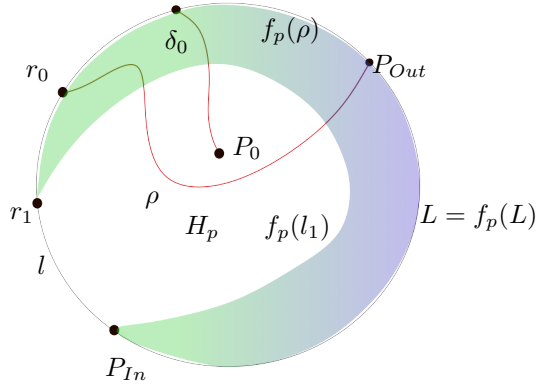

For a proof, see Th.3.15 in \cite{I}. The idea of the proof is fairly simple, and ultimately based on isotoping the first-return map $f_p:\overline{H_p}\setminus\delta\to\overline{D_\alpha}$ to a Smale Horseshoe map by utilizing the Betsvina-Handel algorithm (see \cite{BeH}, and the illustration in Fig.\ref{remove2}). This argument allows one to prove the following fact, a corollary of Th.\ref{th31} which will be utilized throughout this paper (see Cor.3.27 in \cite{I}) - let $g:\overline{H_p}\setminus\overline{\delta}\to\overline{D_\alpha}$ be a homeomorphism isotopic to $f_p:\overline{H_p}\setminus\overline{\delta}\to \overline{D_\alpha}$, s.t. the following is satisfied (see the illustration in Fig.\ref{remove2}):

\begin{itemize}
\item $g(P_{In})=P_{In}$, $g(P_{Out})=P_{Out}$ and $g(L)=L$.
    \item  Whenever $P_0\in\partial H_p$, we have  $\lim_{s\to P_0}g(s)=P_{In}$. 
    \item $P_0\not\in\overline g(\overline{H_p}\setminus\delta)$.
\end{itemize}

Then, we have the following result:

\begin{corollary}
\label{deformation11}    Let $p\in P$ be a trefoil parameter, and let $g:\overline{H_p}\setminus\delta\to\overline{D_\alpha}$ be as above. Then, if $s\in\{1,2\}^\mathbf{N}$ is periodic of minimal period $k$, we have the following:

    \begin{itemize}
        \item There exists $x_s\in D_s$, periodic of minimal period $k$, which is continuously deformed to $y_s$, a periodic orbit of minimal period $k$ for $g$. In particular, $y_s$ is in the invariant set of $g$ in $\overline{H_p}\setminus\overline{\delta}$. That is, the periodic dynamics in $Q$ are unremovable.
        \item Let $s,\omega\in\{1,2\}^\mathbf{N}$ be periodic orbits of minimal periods $k_1$ and $k_2$, and let $x_1\in D_s$ and $x_2\in D_\omega$ be periodic for $f_p$ of minimal periods $k_1$ and $k_2$. If the isotopy deforms $x_i$ to $y_i$, $i=1,2$, then $y_1=y_2$ if and only if $s=\omega$ - that is, the periodic dynamics are also uncollapsible. 
        \item Consequentially, there exists an invariant set $Q'\subseteq \overline{H_p}\setminus\overline\delta$ and a continuous $\phi:Q'\to\{1,2\}^\mathbf{N}$ s.t. $\phi\circ g=\sigma\circ \phi$. Moreover, $\phi(Q')$ includes every periodic symbol in $\{1,2\}^\mathbf{N}$, and $Q'$ includes periodic points of all minimal periods for $g$.
    \end{itemize}
\end{corollary}

\begin{remark}
   The function $g$ from Cor.\ref{deformation11} need not necessarily be a first-return map for a flow. That is, it need only be a homeomorphism of $\overline{H_p}\setminus\delta$.
\end{remark}
\begin{figure}[h]
\centering
\begin{overpic}[width=0.6\textwidth]{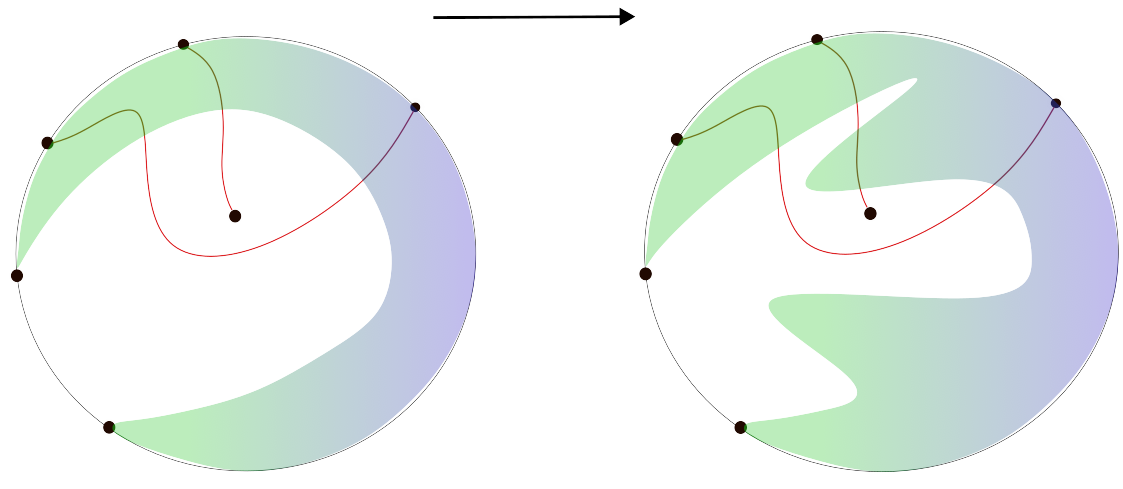}
\put(50,20){$P_{In}$}
\put(230,250){$P_{0}$}
\put(160,290){$\delta$}
\put(-10,300){$r_0$}
\put(545,300){$r_0$}
\put(-100,200){$f_p(r_0)$}
\put(480,200){$g(r_0)$}
\put(380,350){$P_{Out}$}
\put(720,230){$P_0$}
\put(730,300){$\delta$}
\put(600,20){$P_{In}$}
\put(180,150){$H_p$}
\put(420,150){$L$}
\put(-0,150){$l$}
\put(650,190){$H_p$}
\put(1000,150){$L$}
\put(560,150){$l$}
\put(960,360){$P_{Out}$}
\put(790,100){$g(H_p)$}
\put(200,350){$f_p(H_p)$}
\end{overpic}
\caption[An isotopy which preserves $Q$.]{\textit{Since $g$ (on the right) and the first-return map $f_p$ (on the left) are isotopic and $g$ satisfies the assumptions of Cor.\ref{deformation11}, the periodic orbits in $Q$ all persist when $f_p$ is isotoped to $g$.}}
\label{remove2}
\end{figure}
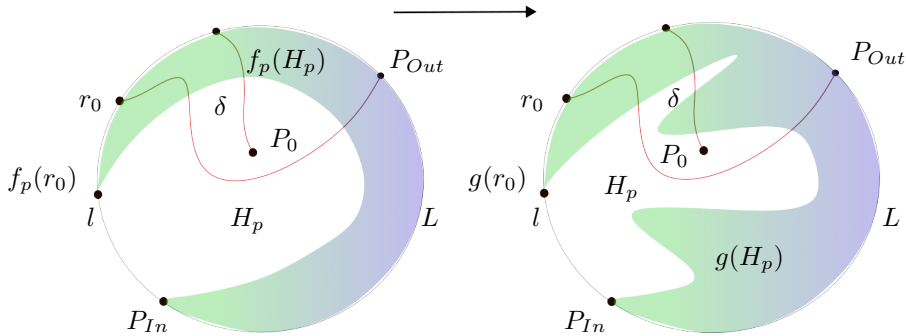
\subsection{Template Theory and the dynamical complexity of three-dimensional flows.}
In addition to the results from \cite{I} listed above, to prove Th.\ref{geometric}, we will also need to several notions from Knot Theory - in particular, we will need the notion of a Template, and its connection to periodic trajectories of flows via the Birman-Williams Theorem (see below). Therefore, we conclude Section 2 with a brief introduction to Template Theory. To begin, let us first recall the notion of a knot:

\begin{definition}
    \label{torusknot} A \textbf{knot} is an isotopy class of an embedding of the circle $S^1$ into $\mathbf{R}^3$. We say two knots in $S^3$, $T_1$ and $T_2$ \textbf{have the same knot type}, provided there exists an isotopy of $\mathbf{R}^3$ which continuously deforms $T_1$ to $T_2$ (see the illustration in Fig.\ref{type}). A knot is said to be a \textbf{Torus Knot} provided it can be looped on a two-dimensional Torus.
\end{definition}

\begin{figure}[h]
\centering
\begin{overpic}[width=0.3\textwidth]{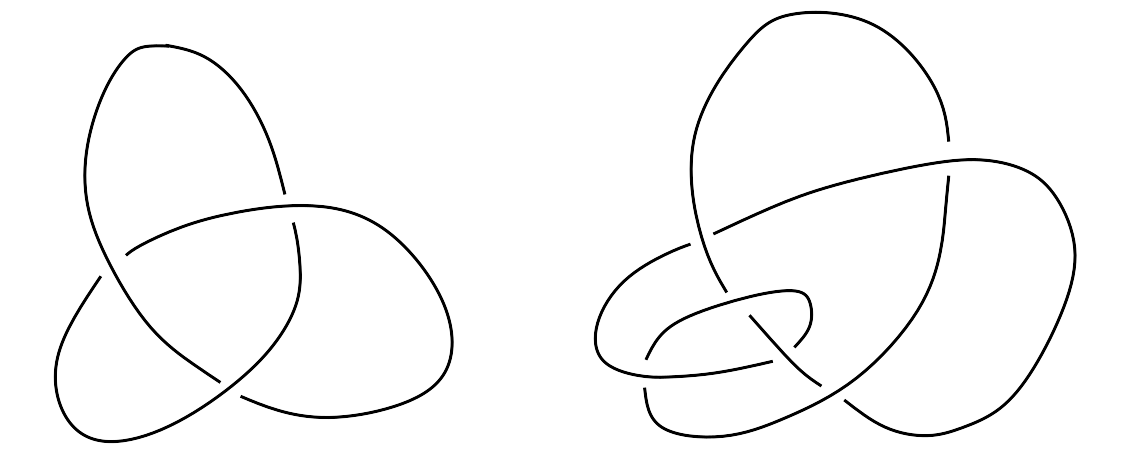}
\end{overpic}
\caption[Two knot types.]{\textit{Two knots, belonging to different knot types. The knot on the left has the same type as the trefoil knot.}}
\label{type}
\end{figure}

It is immediate that given any smooth vector field $F$ on $\mathbf{R}^3$, every periodic trajectory for $\dot{s}=F(s)$ is a knot. It is well-known in the theory of topological dynamics of surface homeomorphisms that one can classify and study surface homeomorphisms (up to isotopy) based on their periodic orbits (for a survey of these results, see \cite{Bo}). Since in this paper we are primarily interested in studying periodic trajectories for flows, let us now introduce the following definition:
\begin{definition}
    \label{dyncomp} Let $F$ be a smooth vector field on $\mathbf{R}^3$ - then, we define its \textbf{dynamical complexity} as the collection of knot types $F$ generates as periodic trajectories (which are not fixed points).
\end{definition}
To see why this definition makes sense, consider two smooth vector fields $F_1$ and $F_2$. If the dynamical complexity of $F_1$ includes different knot types than those in $F_2$, the dynamics of $F_1$ and $F_2$ cannot generate orbitally equivalent flows. As such, provided we can answer which knot types are generated by a given vector field $F$ as periodic trajectories, we essentially answer what makes its dynamics unique.\\

Unfortunately, the problem of classifying the dynamical complexity cannot be solved in full generality for any given smooth vector field on $\mathbf{R}^3$. However, provided we impose a few restriction on the dynamics of the vector field, this question can be answered - an answer we review below. To begin, we first introduce the notion of hyperbolicity for vector fields:

\begin{definition}
    \label{hyperbolicvector} Let $M$ be a smooth $3-$manifold, let $F$ be a $C^\infty$ vector field on $M$, and denote by $\phi_t$ the resulting flow on $M$. An invariant set $\Lambda$ w.r.t. the flow is said to \textbf{have a} \textbf{hyperbolic structure on $\Lambda$} or in short, \textbf{hyperbolic on $\Lambda$}, provided at every $x\in\Lambda$ one can split the tangent space, $T_xM=E^s_x\oplus E^u_x\oplus E^c(x)$ s.t. the following is satisfied:
    \begin{itemize}
        \item $E^c(x)$ is spanned by $F(x)$.
        \item $E^s_x,E^u_x$ and $E^c_x$ vary continuously to $E^s_{\phi_t(x)},E^u_{\phi_t(x)}$ and $E^c_{\phi_t(x)}$ (respectively) as $x$ flows to $\phi_t(x)$, $t\in\mathbf{R}$ - in particular, $D_{\phi_t}(x)E^j_x=E^j_{\phi_t(x)}$ where $j\in\{s,u,c\}$ and $D_{\phi_t}(x)$ denotes the differential of the flow at time $t$.
        \item There exist constants $C>0$ and $\lambda>1$ s.t. for every $t>0$ and every $x\in\Lambda$ we have: 
        
        \begin{enumerate}
            \item For $v\in E^s_x$, $||D_{\phi_t}(x)v||<Ce^{-\lambda t}||v||$.
            \item For $v\in E^u_x$, $||D_{\phi_t}(x)v||>Ce^{\lambda t}||v||$. 
        \end{enumerate}
    \end{itemize}

In other words, the flow expands (or stretches) in one direction along the trajectories on $\Lambda$, and contracts in another direction. 
\end{definition}
We will also need the following definition:

\begin{definition}
    \label{nonwanderchain}
    Let $M$ be a smooth $3-$manifold, let $F$ be a $C^\infty$ vector field on $M$, and denote by $\phi_t$ the resulting flow on $M$. We define the \textbf{non-wandering set in $M$}, as the maximal invariant set of $\phi_t$ in $M$.
    \end{definition}
When $M$ is a manifold with a boundary and trajectories can escape $M$ in finite time under the flow, the non-wandering set includes the maximal set of initial conditions which never escape $M$ (in particular, it includes every periodic trajectory in $M$).\\

\begin{figure}[h]
\centering
\begin{overpic}[width=0.3\textwidth]{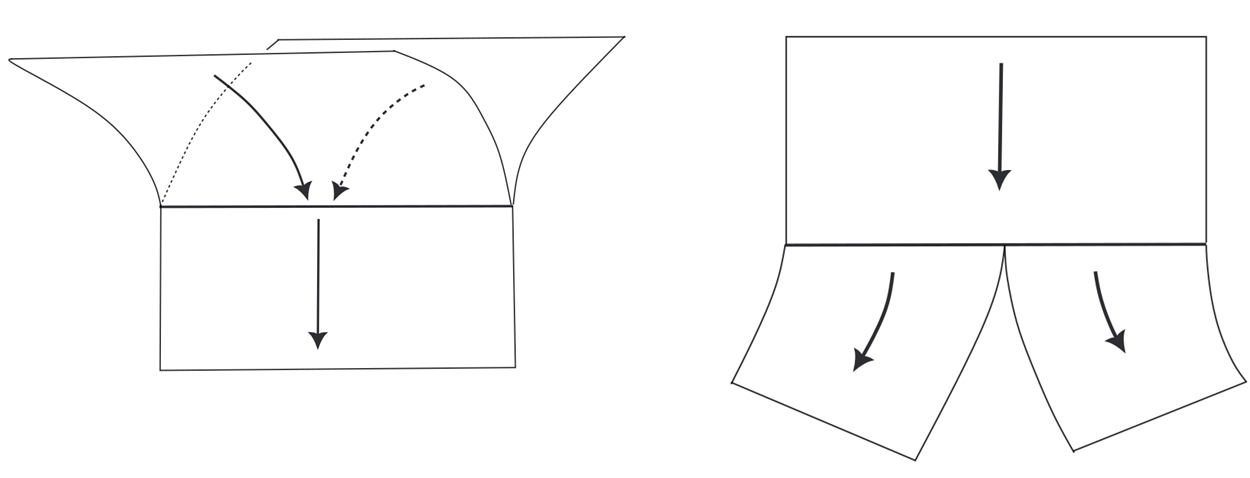}
\end{overpic}
\caption[Joining and splitting charts.]{\textit{A joining chart on the left, and a splitting chart on the right.}}
\label{TEMP1}
\end{figure}

As it turns out, given a $3-$manifold $M$ endowed with a flow $\phi_t$ hyperbolic on its non-wandering set in $M$, one can describe the dynamical complexity of $\phi_t$ in $M$ - i.e., we can answer the question which knot types $\phi_t$ can and cannot generate as periodic trajectories. This is done by reducing the flow to an object called a Template, per the Birman and Williams Theorem (see Th.\ref{BIRW} below). We first define:

\begin{definition}\label{deftemp}
    A \textbf{Template} is a compact branched surface with a boundary endowed with a smooth expansive semiflow, built locally from two types of charts - \textbf{joining} and \textbf{splitting} (see the illustration in Fig.\ref{TEMP1}). Additionally, the gluing maps between charts all respect the semiflow, and act linearily on the edges (see Fig.\ref{TEMP2} for illustrations).
\end{definition}

\begin{figure}[h]
\centering
\begin{overpic}[width=0.2\textwidth]{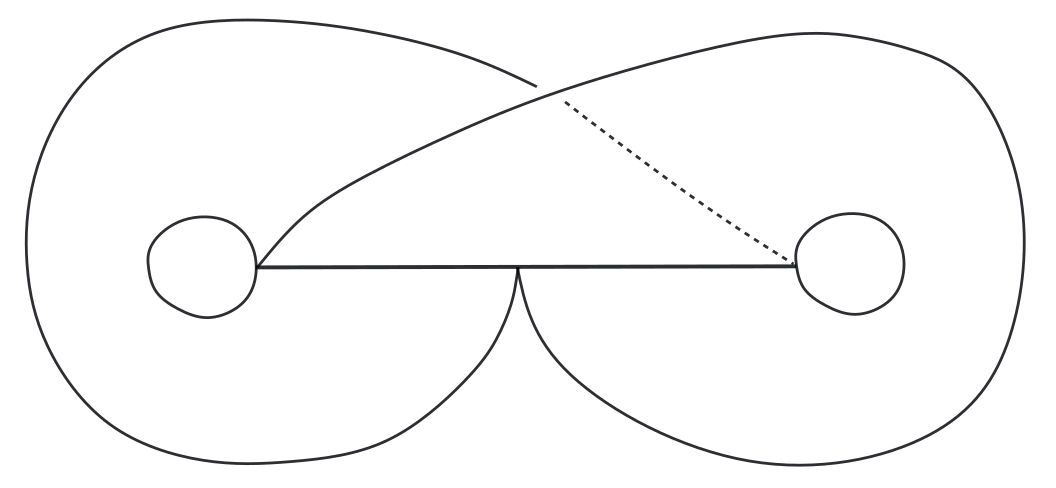}
\end{overpic}
\caption[A Template.]{\textit{A Template.}}
\label{TEMP2}
\end{figure}

Now, let $M$ be a three-manifold, and let $\phi_t$ be flow with a hyperbolic non-wandering set $\Lambda\subseteq M$ ( for example, consider a flow generated by suspending a Smale Horseshoe). Moreover, let $P(\Lambda)$ denote the periodic trajectories for $\phi_t$ in $\Lambda$. Then, under these assumptions, we have the following result:
\begin{claim}{\textbf{The Birman-Williams Theorem} -}\label{BIRW}
    Given $M,\phi_t,\Lambda$ as above, there exists a unique Template $T$, embedded in $M$ and endowed with a smooth semiflow $\psi_t$. Moreover, letting $P(T)$ denote the periodic trajectories of $\psi_t$ in $T$, there exists an injective projection $f:P(\Lambda)\to P(T)$ s.t. $P(T)\setminus f(P(\Lambda))$ contains at most two periodic trajectories for $\psi_t$. Moreover, $T=\overline{P(T)}$.
\end{claim}
For a proof, see either \cite{BW} or Th.2.2.4 in \cite{KNOTBOOK}. In other words, the Birman-Williams Theorem states that whenever we have a flow hyperbolic on some invariant set $\Lambda$, the dynamical complexity of the flow on $\Lambda$ is encoded by some Template, which dictates which knot types exist in $\Lambda$ (save possibly for two).
\section{Constructing a dynamically minimal flow}

In this section we prove Th.\ref{deform} and Th.\ref{geometric} - namely, in this section we prove the dynamics of the Rössler system at trefoil parameters can be reduced to those of an explicit flow, hyperbolic on its non-wandering set. We do so by proving Th.\ref{deform} - whose proof would consist this entire section.\\

The idea behind the proof is as follows - using Th.\ref{th31} we construct an explicit, smooth deformation of the vector field $F_p$ to $G$ - a smooth vector field on $\mathbf{R}^3$, hyperbolic on its non-wandering set. As will be made clear, the flow generated by $G$ does not contain "extra"-dynamical information - or, more precisely, every periodic trajectory for $G$ (save perhaps for one), can be deformed by an isotopy of $\mathbf{R}^3$ to a periodic trajectory for $F_p$.

\begin{theorem}
\label{deform}    Let $p\in P$ be a trefoil parameter for the Rössler system, and recall we denote the corresponding vector field by $F_p$. Then, there exists a smooth vector field on $S^3$, $G$, s.t. the following is satisfied: 
    \begin{itemize}

        \item The vector field $F_p$ can be smoothly deformed to $G$ in $\mathbf{R}^3$.
        \item $G$ has a first-return map which is a Smale Horseshoe map.
        \item The knot-types for the periodic trajectories generated by $F_p$ include every knot type generated by $G$ as a periodic trajectory (save possibly for two knot-types).
    \end{itemize}
\end{theorem}
\begin{proof}
We prove Th.\ref{deform} by constructing the deformation from $F_p$ to $G$ explicitly, using an argument which closely mirrors the one used to prove Th.\ref{th31}. We do so in several steps, as outlined below:
    
    \begin{itemize}
        \item First, we deform $F_p$ to the vector field $R_p$ given by Th.\ref{th21} - this allows us to extend $F_p$ to a smooth vector field of $S^3$, without losing any of the periodic trajectories given by Th.\ref{th31}.
        \item Second, we deform $R_p$ to $G_p$ by opening up the fixed-points by Hopf-bifurcations. As will be made clear, this will allow us to describe the first-return map for $G_p$ as a "distorted" Smale Horseshoe.
        \item Finally, we deform $G_p$ to $G$ by "straightening" its dynamics - that is, if $G_p$ generates a flow which is that of a suspended, distorted horseshoe, we deform the said flow by removing that said "distortions" - from which Th.\ref{deform} would follow. 
\end{itemize}
\subsection{Stage $I$ - deforming $F_p$ to $R_p$.}
Recall the vector field $R_p$ from Th.\ref{th21} - as proven in Th.\ref{th21}, given any sufficiently large radius $r>0$ we can choose $R_p$ s.t. it coincides with the vector field $F_p$ on $B_r(0)$. As a consequence, by the boundedness of $B_\alpha$ we can smoothly deform $F_p$ to $R_p$ without destroying any of the dynamics of $F_p$ inside $B_\alpha$.\\
\begin{figure}[h]
\centering
\begin{overpic}[width=0.65\textwidth]{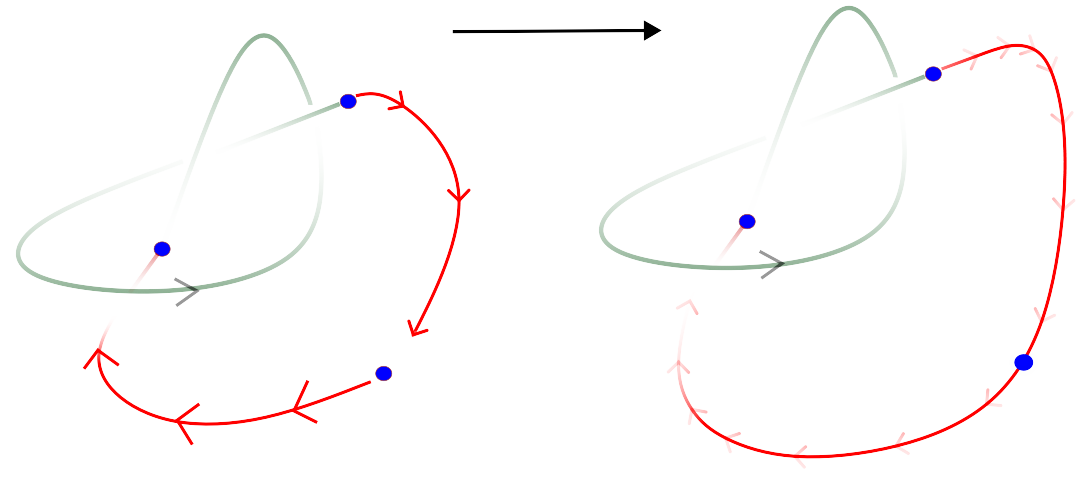}%
\put(930,280){$\Gamma$}
\put(840,410){$P_{Out}$}
\put(320,380){$P_{Out}$}
\put(390,80){$\infty$}
\put(60,200){$P_{In}$}
\put(600,230){$P_{In}$}
\put(970,120){$\infty$}
\put(420,180){$\Gamma_{Out}$}
\put(250,110){$\Gamma_{In}$}
\end{overpic}
\caption[Deforming $F_p$ to $R_p$ for trefoil parameters.]{\textit{The deformation of $F_p$ to $R_p$ - $\Gamma_{Out},\Gamma_{In}$ are connected by the deformation to a heteroclinic trajectory $\Gamma$ (thus destroying the fixed point at $\infty$).}}\label{TRANS}
\end{figure}
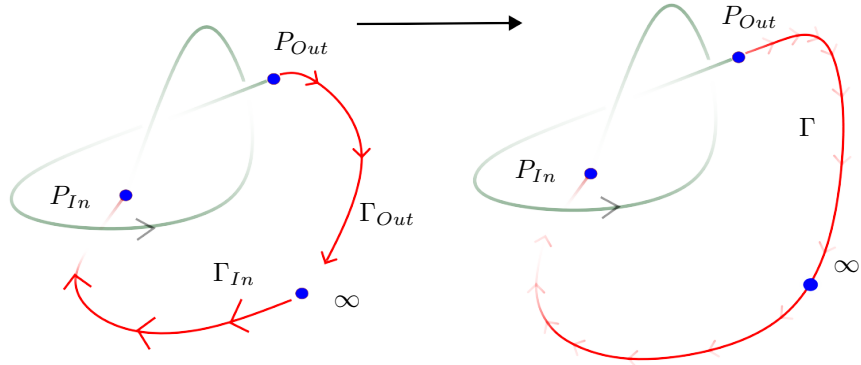

Since $\overline{B_\alpha}\subseteq B_r(0)$, the topological ball $B_\alpha$ persists as $F_p$ is smoothly deformed to $R_p$. Therefore, since $B_\alpha$ is invariant under $F_p$, it remains so under $R_p$ as well. Let us now consider the cross-section $U_p$, and recall it is defined as the maximal subset on $\{\dot{y}=0\}$ on which trajectories cross from $\{\dot{y}>0\}$ to $\{\dot{y}<0\}$. Since $U_p$ is a topological disc, we construct the deformation from $F_p$ to $R_p$ s.t. $U_p$ persists as well. That is, we construct the smooth deformation from $F_p$ to $R_p$ s.t. $U_p$ both remains the maximal set on $\{\dot{y}=0\}$ at which trajectories cross from $\{\dot{y}>0\}$ to $\{\dot{y}<0\}$, and a half-plane. Now, using Def.\eqref{def32} we conclude:

\begin{corollary}
\label{coin}    If $p\in P$ is a trefoil parameter, then the one-dimensional invariant manifolds for $P_{In}$ and $P_{Out}$ coincide to a one-dimensional manifold $W$ - in particular, $W=\Theta\cup\Gamma$ (see the illustration in Fig.\ref{TRANS}). Moreover, $W$ is the one-dimensional stable manifold for $P_{In}$, and the one-dimensional unstable manifold of $P_{Out}$ (w.r.t. $R_p$).
\end{corollary}

Cor.\ref{coin} essentially implies the following - given an initial condition $s\in\mathbf{R}^3\setminus W$, neither its forward nor backwards trajectory under $R_p$ can tend to $\infty$. With this idea in mind, let $r_p:\overline{U_p}\setminus W\to\overline{U_p}\setminus W$ denote the first-return map w.r.t. $R_p$ - wherever defined in $\overline{U_p}\setminus W$. We now prove:

\begin{lemma}\label{corhp}
Let $p\in P$ be a trefoil parameter, and let $W$ be as in Cor.\ref{coin}. Then, $R_p$ and $r_p$ satisfy the following:

\begin{itemize}

    \item Given any initial condition $s\in\mathbf{R}^3\setminus W$ which is not a fixed point, the trajectory of $s$ w.r.t. $R_p$ eventually hits $U_{p}$ transversely.
    \item The first return map $r_p:\overline{U_p}\setminus W\to\overline{U_p}\setminus W$ is well-defined. 
    \item We have $r_p|_{\overline{D_\alpha}\setminus \{P_0\}}=f_p|_{\overline{D_\alpha}\setminus \{P_0\}}$ - that is, the first-return maps $f_p$ for $F_p$ and $r_p$ for $R_p$ coincide on the $\overline{D_\alpha}\setminus \{P_0\}$ (they possibly differ on the rest of $U_p$).
\end{itemize}

\end{lemma}
\begin{proof}
Given $s\in\mathbf{R}^3$, denote by $\gamma_s$ its solution curve w.r.t. $R_p$, and parameterize it s.t. $\gamma_s(0)=s$ - moreover, denote by $\gamma_s(0,\infty)$ the forward trajectory of $s$. Now,  recall that by definition, the topological disc $U_{p}$ is the maximal set on $\{\dot{y}=0\}$ on which $R_p$ points into $\{\dot{y}<0\}$. As such, we conclude that if $\dot{y}$ changes sign along $\gamma_s(0,\infty)$ at least twice, $\gamma_s(0,\infty)$ must intersect $U_{p}$ transversely.\\

We now prove that given $s\in\mathbf{R}^3\setminus W$ which is not a fixed point, its forward trajectory hits $U_{p}$ transversely infinitely many times. By the discussion above, it would suffice to prove the sign of $\dot{y}$ changes infinitely many times along $\gamma_s$. We do so at two stages - first, we prove this is the case for $s\in\mathbf{R}^3$ whose forward trajectory does not tend to the fixed points $P_{In}$ and $P_{Out}$. To do so, let us first note that since $\Gamma\subseteq W$, by $s\not\in W$ it follows the forward trajectory of $s$ cannot tend to $\infty$ - which implies that because $\gamma_s(0,\infty)$ does not limit to either saddle foci $P_{In},P_{Out}$ in forward time, the sign of $\dot{y}$ (w.r.t. $R_p$) along $\gamma_s(0,\infty)$ must change infinitely many times. Therefore, whenever the forward trajectory of $s$ does not tend to either $P_{In}$ or $P_{Out}$ it intersects transversely with $U_{p}$ infinitely many times.\\

Now, recall $W=\Theta\cup\Gamma$, and that both $\Gamma$ and $\Theta$ are heteroclinic trajectories directed from $P_{Out}$ to $P_{In}$. Recalling $P_{In}$ is a saddle-focus with a one-dimensional stable manifold, it is immediate $W$ is the said one-dimensional stable manifold of $P_{In}$. As such, we conclude that whenever $s\not\in W$, its forward trajectory cannot tend to $P_{In}$. As far as the saddle-focus $P_{Out}$ is concerned, recall its stable manifold, $W^s_{Out}$, is two-dimensional (see the discussion at page \pageref{eq:9}), and that is its transverse to $U_p$ around $P_{Out}$ - w.r.t. the vector field $F_p$ (see Lemma \ref{TR}). Since $F_p$ and $R_p$ coincide around $P_{In}$ and $P_{Out}$ (as both are interior to $B_r(0)$ - see the discussion before Lemma \ref{coin}), the same is true for $R_p$ - that is, if $W'$ is the stable manifold for $P_{Out}$ w.r.t. $R_p$, $W'$ is transverse to $U_p$ at $P_{Out}$ (see Fig.\ref{loci}). Consequentially, since $P_{Out}$ is a saddle-focus, given any $s\in W'$ its forward trajectory eventually flows to some neighborhood of $P_{Out}$ where it hits $U_{p}$ transversely infinitely many times.\\

All in all, given any $s\in\mathbf{R}^3\setminus W$ that is not a fixed point, because $s$ cannot lie on either $\Gamma$ or $\Theta$ (i.e., its trajectory cannot tend to $P_{In}$), there are precisely two possibilities: either the forward trajectory of $s$ lies on $W'$, or it does not. If $s\not\in W'$, its forward trajectory hits $U_p$ transversely infinitely many times - and by the discussion above, the same is true whenever $s\in W'$. Therefore, given any initial condition $s\in\mathbf{R}^3$ that is not a fixed point and does not lie in $W$, its forward trajectory would eventually hit $U_p$ transversely. As a consequence, $U_p$ is a universal cross-section for $R_p$ - and moreover, that first-return map $r_p:\overline{U_p}\setminus W\to\overline{U_p}\setminus W$ is well-defined.\\

Finally, let us consider the closed ball $\overline{B_\alpha}\subseteq B_r(0)$ - by the construction of $R_p$ it follows that because $F_p$ and $R_p$ coincide on $B_r(0)$, they also coincide on any invariant sets inside $B_r(0)$. Let us recall $\overline{B_\alpha}$ is $F_p$-invariant, hence the first-return maps $r_p$ and $f_p$ coincide on the set $\overline{B_\alpha}\setminus W\cap \overline{U_p}$. Now, let us recall $\overline{B_\alpha}\cap \overline{U_p}$ is the disc $\overline{D_\alpha}$ (see Prop.\ref{lem33}), therefore, $f_p$ and $r_p$ coincide on $\overline{D_\alpha}\setminus W$. Since by Def.\ref{def32} we know $W\cap \overline{D_\alpha}=\{P_0\}$, Lemma \ref{corhp} now follows.
\end{proof}

\subsection{Stage $II$ - deforming $R_p$ to $G_p$.}    
Having smoothly deformed $F_p$ to $R_p$ we now continue our analysis - and in this stage we smoothly deform $R_p$ to a $C^\infty$ vector field on $S^3$, $G_p$, whose dynamics include a suspended topological horseshoe. To begin let us perform Hopf-Bifurcations at both fixed-points $P_{In}$ and $P_{Out}$, replacing them by a sink $P'_{In}$ and a source $P'_{Out}$, respectively (see the illustration in Fig.\ref{HOPFA}) - this smooth deformation expands up the set $W=\Theta\cup\Gamma$ to an knotted tube, as indicated in Fig.\ref{HOPFBB}.\\ 

\begin{figure}[h]
\centering
\begin{overpic}[width=0.45\textwidth]{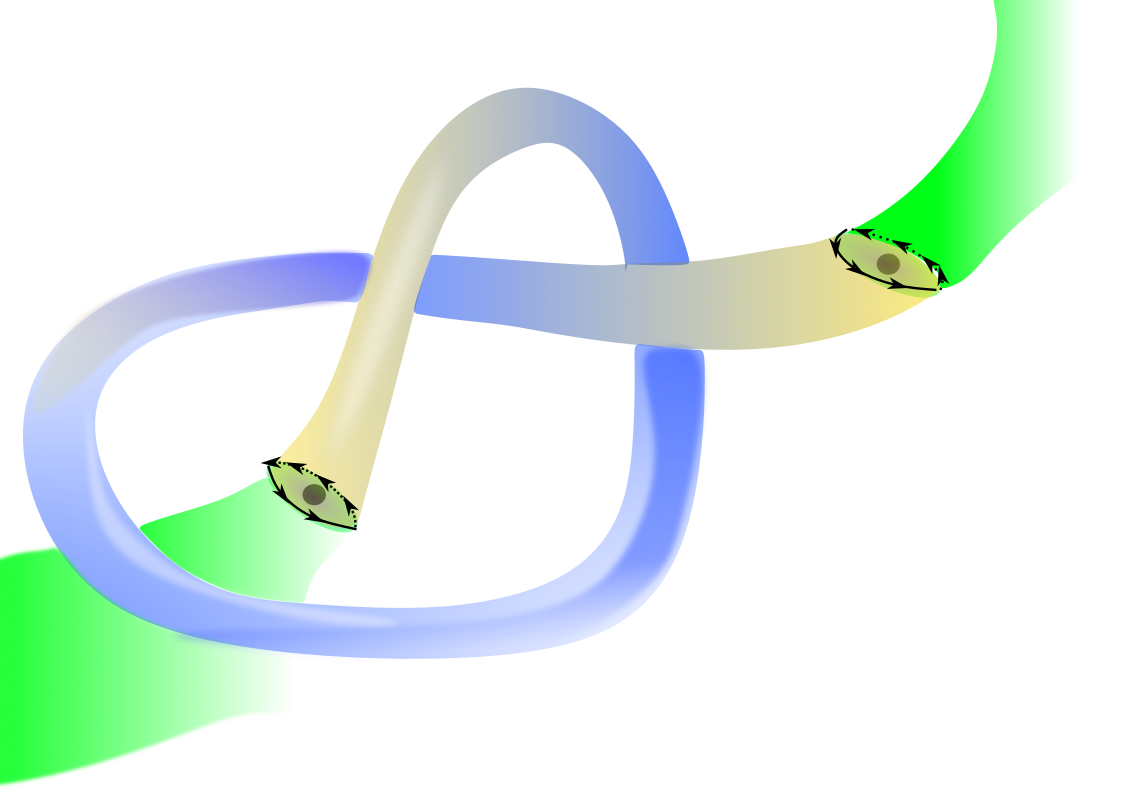}
\put(840,390){$P_{Out}$}
\put(350,220){$P_{In}$}
\end{overpic}
\caption[Expanding the heteroclinic trajectories to tubes.]{\textit{The effects of the Hopf bifurcations - $P_{In}$ and $P_{Out}$ become periodic orbits surrounding the sink and the source $P'_{In}$ and $P'_{Out}$ (respectively), while the heteroclinic trajectories are expanded from curves into tubes. The blue-yellow tube denotes the blow-up of $\Theta$, while the green one denotes the blow-up of $\Gamma$.}}\label{HOPFBB}
\end{figure}
Since $W$ becomes a tube under this deformation, it follows the points in $W\cap U_p$ all becomes open discs - in particular, the point $P_0=\Theta\cap\overline{U_p}$ is opened into an open disc (see the illustration in Fig.\ref{HOPFA}). In addition, we choose this deformation the first-return map of $R_p$, $r_p:\overline{U_p}\setminus W\to\overline{U_p}\setminus W$ is isotopically deformed to $g_p:\overline{U_p}\setminus W\to\overline{U_p}\setminus W$, the corresponding first return map for $G_p$. In particular, we choose $G_p$ s.t. the following is satisfied (see Fig.\ref{HOPFA}):

\begin{itemize}
    \item $P'_{In}$ is an attracting fixed point on $\partial U_p$, while $P'_{Out}$ is a repeller on  $\partial U_p$. Moreover, $P_{In},P_{Out}$ are hyperbolic fixed points for $g_p$, strictly interior to the cross-section $U_p$.
    \item The set $D_\alpha$ is preserved - that is, we deform $R_p$ to $G_p$ s.t. the two dimensional unstable manifold of $P_{In}$ persists as the two-dimensional stable manifold of $P_{Out}$. That is, we do not destroy the topological ball $B_\alpha$.
    \item The basin of attraction of $P'_{In}$ on $U_p$ includes a half-disc $D_{In}$ s.t. $P'_{In},P_{In}\in\partial D_{In}$ - see the illustration in Fig.\ref{HOPFA}.
    \item We deform $R_p$ to $G_p$ s.t. $G_p$ s.t. both $g_p(P_0)\subseteq D_{In}$, and $g_p$ is continuous around the closed disc $\overline{P_0}$ (see the illustration in Fig.\ref{HOPFBB}). In more detail, similarly to what we did in the proof of Th.\ref{th31} we open the curve $\delta$ given by Prop.\ref{lem33} to an open sector $\delta'$ containing $P_0$, as illustrated in Fig.\ref{HOPFA} and Fig.\ref{HOPFBB} - and we do so s.t. $g_p$ is continuous on $\overline{\delta'}$, and $\delta'\cap\overline{H_p}\subseteq\partial H_p$ (as in Fig.\ref{HOPFA}).
\end{itemize}

 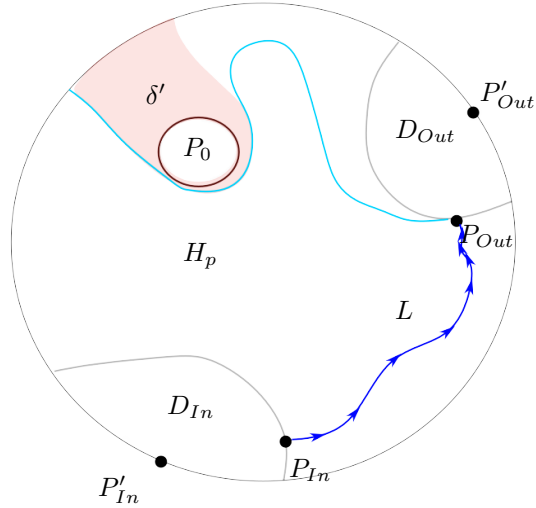
\begin{figure}[h]
\centering
\begin{overpic}[width=0.4\textwidth]{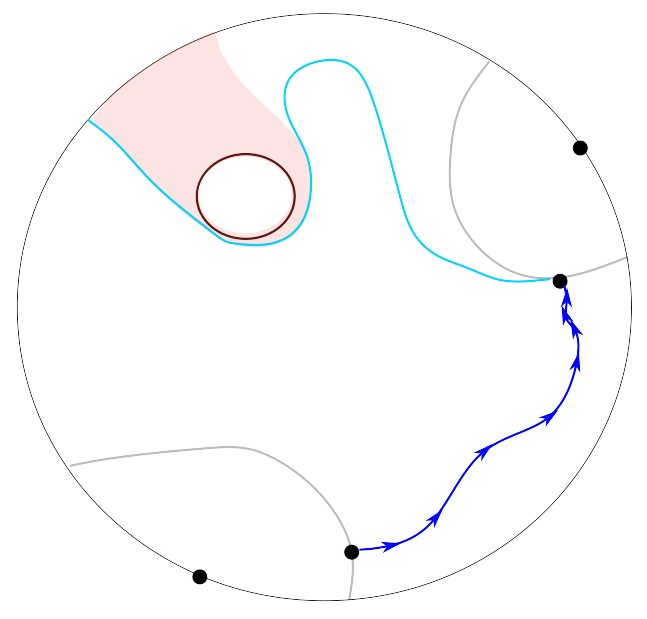}
\put(190,0){$P'_{In}$}
\put(550,35){$P_{In}$}
\put(750,335){$L$}
\put(320,155){$D_{In}$}
\put(750,680){$D_{Out}$}
\put(350,645){$P_0$}
\put(280,745){$\delta'$}
\put(910,750){$P'_{Out}$}
\put(870,480){$P_{Out}$}
\put(350,445){$H_p$}
\end{overpic}
\caption[The effects of the Hopf bifurcation on $U_p$.]{\textit{The Hopf bifurcation opens $P_{In}$ and $P_{Out}$ to periodic orbits, while $P_0$ is opened to a disc. $D_{In}$ denotes the immediate basin of attraction for a sink $P'_{In}$ while $D_{Out}$ denotes the immediate basin of repelling for the source $P'_{Out}$. $L$ persists under the deformation, while the red region is $\delta'$, the sector generated by opening $\delta$. $H_p$ is the region trapped between $L$ and $\partial U_p$.}}
\label{HOPFA}
\end{figure}

Now, consider the cross-section $H_p\subseteq D_\alpha\subseteq U_p$ (see Prop.\ref{lem33}), and recall that by Lemma \ref{corhp} we have $f_p|_{\overline{D_\alpha}\setminus\{P_0\}}=r_p|_{\overline{D_\alpha}\setminus\{P_0\}}$ - where $f_p$ is the first-return map for the original vector field, $F_p$ (i.e., for the original Rössler system at trefoil parameters), while $r_p$ is the first-return map for $R_p$. This allows us to choose the deformation of $F_p$ to $G_p$ s.t. the set $H_p$ persists as a topological disc inside $U_p$ (see the illustration in Fig.\ref{HOPFA}). Therefore, as we deform deform $F_p$ smoothly to $R_p$ and then to $G_p$, by the construction of $G_p$ outlined above it follows the deformation of $F_p$ to $G_p$ induces an isotopy between $r_p:\overline{H_p}\setminus\delta\to\overline{U_p}$ and $g_{p}:\overline{H_p}\setminus\delta\to\overline{U_p}$ - as illustrated in Fig.\ref{HOPFB} and Fig.\ref{HOPFBBB}. \\

 \begin{figure}[h]
\centering
\begin{overpic}[width=0.4\textwidth]{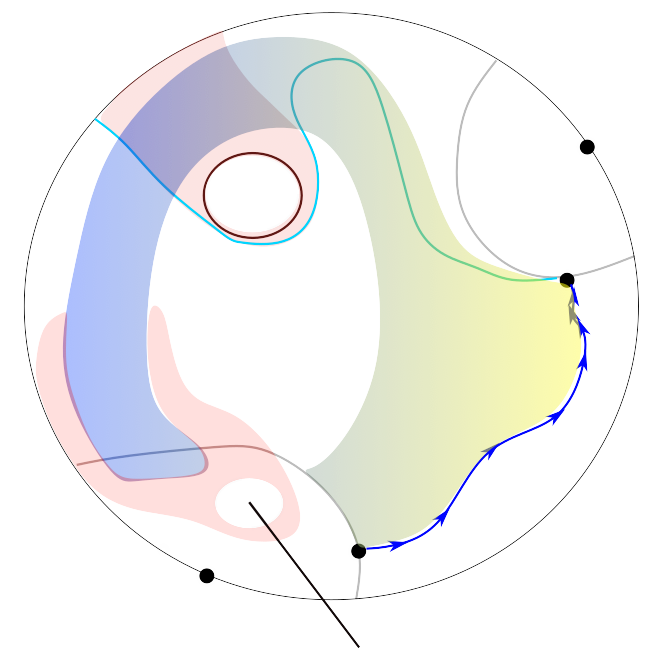}
\put(190,70){$P'_{In}$}
\put(550,115){$P_{In}$}
\put(550,5){$g_p(P_0)$}
\put(750,265){$L$}
\put(320,155){$D_{In}$}
\put(750,680){$D_{Out}$}
\put(350,705){$P_0$}
\put(280,775){$\delta'$}
\put(900,790){$P'_{Out}$}
\put(870,520){$P_{Out}$}
\put(630,480){$g_p(H_p)$}
\put(350,445){$H_p$}
\end{overpic}
\caption[The image of $H_p$ under $g_p$.]{\textit{The image of $H_p$ under $g_p$ is sketched as the yellow-blue region. As can be seen, $g_p(P_0)\subseteq D_{In}$, while $g_p$ is continuous around $\delta'$.}}
\label{HOPFB}
\end{figure}
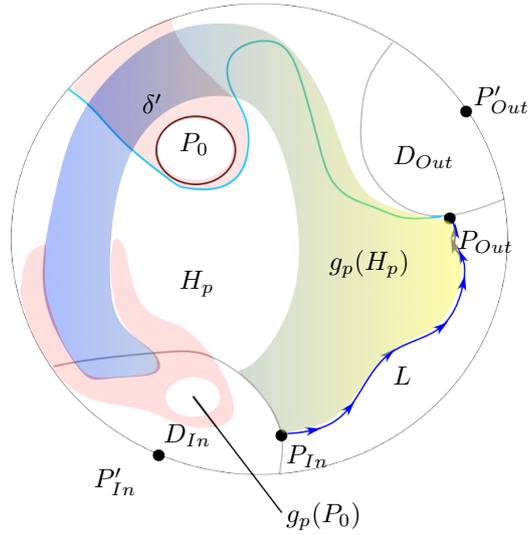

Now, recall that per Cor.\ref{deformation11} the periodic orbits for $f_p$ in the set $Q\subseteq H_p$ are both isotopy stable and uncollapsible. Since neither $f_p(H_p)$ or $g_p(H_p)$ intersect $P_0$ (see the illustration in Fig.\ref{HOPFBBB}), by Cor.\ref{deformation11} we know every periodic orbit for $f_p$ in $Q$ is isotopically deformed to a unique periodic orbit for $g_p$ in $Q$ - of the same minimal period. Therefore, because $F_p$ is smoothly deformed to $G_p$ this immediately implies:

\begin{corollary}
\label{corhalf}    Let $p\in P$ be a trefoil parameter, and let $F_p,G_p, U_p$ and $Q$ be as above. Let $T_0$ be a periodic trajectory for $F_p$ s.t. $T_0\cap U_p\subseteq Q$. Then, as we smoothly deform $F_p$ to $G_p$, $T_0$ is isotopically deformed to $T_\frac{1}{2}$, a periodic trajectory for $G_p$, s.t. $T_\frac{1}{2}\cap U_p\subseteq H_p$. Consequentially, $T_0$ and $T_\frac{1}{2}$ have the same knot type, while $T_0\cap U_p$ and $T_\frac{1}{2}\cap U_p$ have the same cardinality.
\end{corollary}

\begin{figure}[h]
\centering
\begin{overpic}[width=0.75\textwidth]{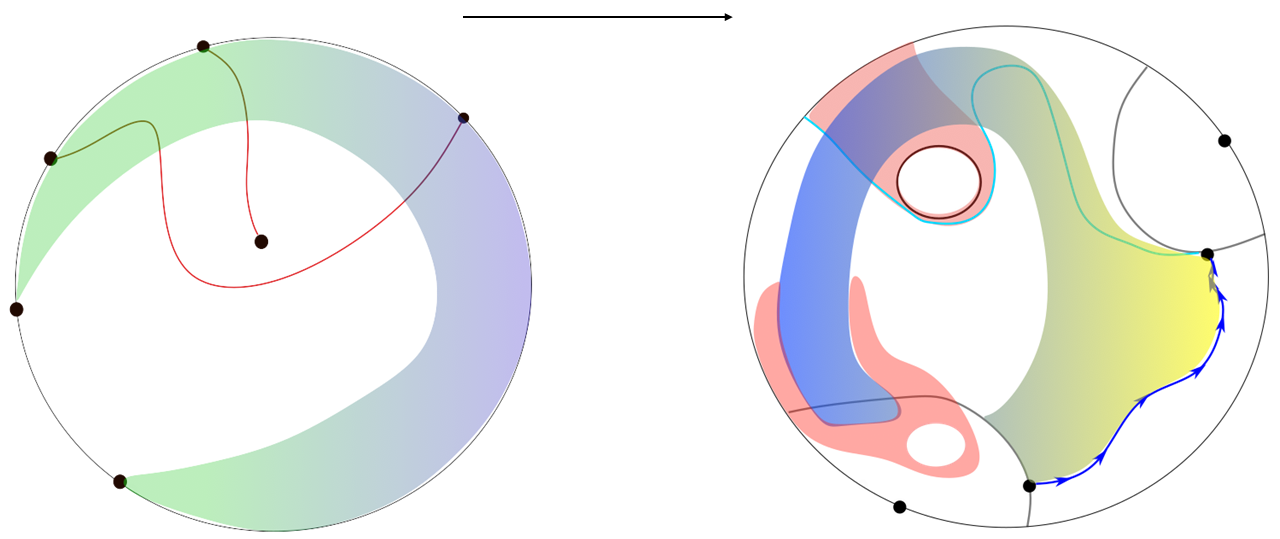}%
\put(880,280){$D_{Out}$}
\put(940,330){$P'_{Out}$}
\put(380,330){$P_{Out}$}
\put(150,280){$\delta$}
\put(680,350){$\delta'$}
\put(60,-10){$P_{In}$}
\put(810,10){$P_{In}$}
\put(650,-10){$P'_{In}$}
\put(960,190){$P_{Out}$}
\put(180,180){$P_0$}
\put(710,270){$P_0$}
\put(250,60){$f_p(H_p)$}
\put(830,150){$g_p(H_p)$}
\end{overpic}
\caption[Isotoping $f_p$ to $g_p$.]{\textit{The isotopy deforming $f_p$ to $g_p$, induced by the smooth deformation of the vector field $F_p$ to $G_p$. }}\label{HOPFBBB}
\end{figure}
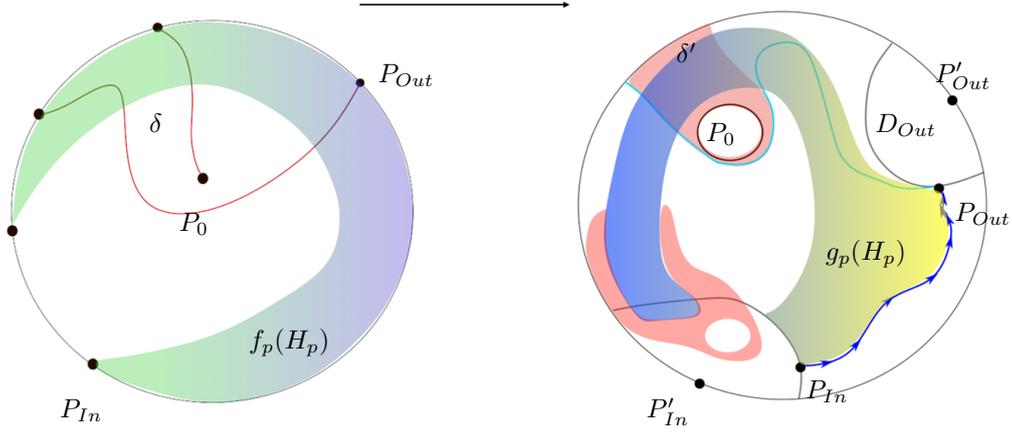

To continue, recall we constructed $G_p$ s.t. $g_p$ is continuous around the closed disc $\overline{P_0}$, and, more generally, on the closed sector $\overline{\delta'}$ as defined above (see the illustrations in Fig.\ref{HOPFA} and \ref{HOPFBBB}) - in particular, recall $g_p(P_0)\subseteq D_{In}$. Let us now denote by $AB$ an arc on $\partial D_{In}$ with one end at $P_{In}$ as in Fig.\ref{HOPFD}, and by $CD$ some arc in $\partial H_p$, s.t. the following is satisfied:
\begin{itemize}
    \item $g_p(H_p)\setminus CD$ includes (at least) two components connecting $AB$ and $CD$ (as illustrated in Fig.\ref{HOPFD}).
    \item $H_p$ and $P_0$ lie in different components of $U_p\setminus CD$.
    \item By smoothly deforming $G_p$ (if necessary), we may also arrange $g_p(CD)\subseteq D_{In}$.
\end{itemize}

 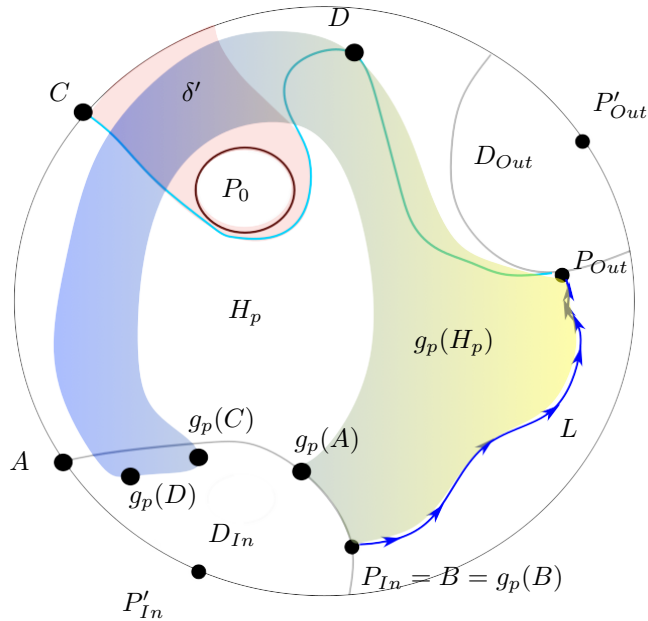
\begin{figure}[h]
\centering
\begin{overpic}[width=0.5\textwidth]{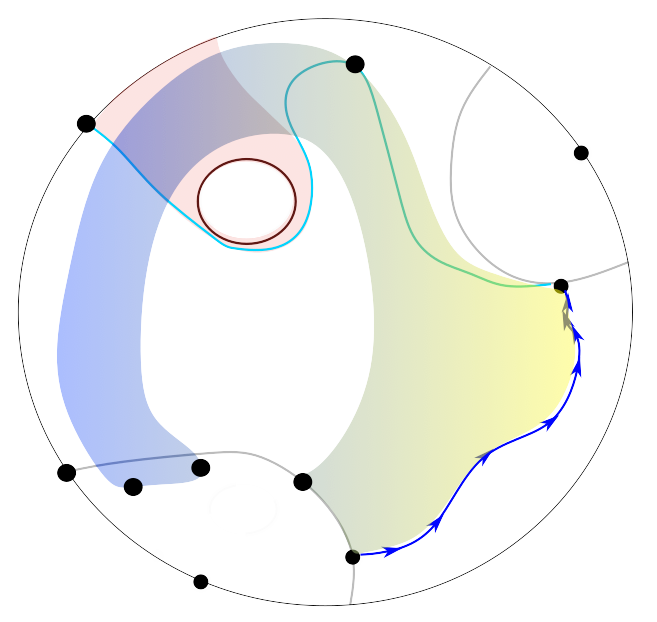}
\put(190,0){$P'_{In}$}
\put(550,40){$P_{In}=B=g_p(B)$}
\put(450,250){$g_p(A)$}
\put(290,280){$g_p(C)$}
\put(200,165){$g_p(D)$}
\put(850,265){$L$}
\put(320,105){$D_{In}$}
\put(720,680){$D_{Out}$}
\put(340,630){$P_0$}
\put(280,775){$\delta'$}
\put(80,775){$C$}
\put(500,890){$D$}
\put(20,215){$A$}
\put(900,760){$P'_{Out}$}
\put(870,520){$P_{Out}$}
\put(630,400){$g_p(H_p)$}
\put(350,445){$H_p$}
\end{overpic}
\caption[Choosing sides in $H_p$.]{\textit{Choosing the $AB$ and $CD$ sides in $\partial H_p$.}}
\label{HOPFD}
\end{figure}

By considering the $AB$ and $CD$ arcs on $\partial H_p$ we can now denote $H_p$ as $ABCD$, a topological rectangle. Consequentially, we now sketch $g_{p}$ as a rectangle map, $g_{p}:ABCD\to\overline{U_p}$, as appears in Fig.\ref{HOPFD1}. As can be seen in Fig.\ref{HOPFD1}, $g_p$ looks very similar to a Smale Horseshoe map - in the sense that even though it may be more complicated than a Smale Horseshoe map, it still appears to stretch and fold $H_p=ABCD$ very similarly to a Smale Horseshoe map. We now make this idea precise by proving the following fact:
\begin{proposition}
\label{doubla}    There exists an invariant set $I$ in $\overline{D_\alpha}$ and a continuous, surjective map $\xi:I\to\{1,2\}^\mathbf{Z}$ s.t. $\xi\circ g_{p}=\zeta\circ\xi$, where $\zeta:\{1,2\}^\mathbf{Z}\to\{1,2\}^\mathbf{Z}$ is the double-sided shift. Additionally, we have:

\begin{itemize}
    \item Given any $s\in\{1,2\}^\mathbf{Z}$, $\zeta^{-1}(s)$ is contractible.
    \item If $s\in\{1,2\}^\mathbf{Z}$ is periodic of minimal period $k$, $\xi^{-1}(s)$ includes a periodic point $x_s$ for $g_p$ of minimal period $k$.
    \item $x_s$ can be isotopically deformed to a periodic point of $f_p$ in $\overline{H_p}$ of minimal period $k$.
\end{itemize}

\end{proposition}

 \begin{figure}[h]
\centering
\begin{overpic}[width=0.4\textwidth]{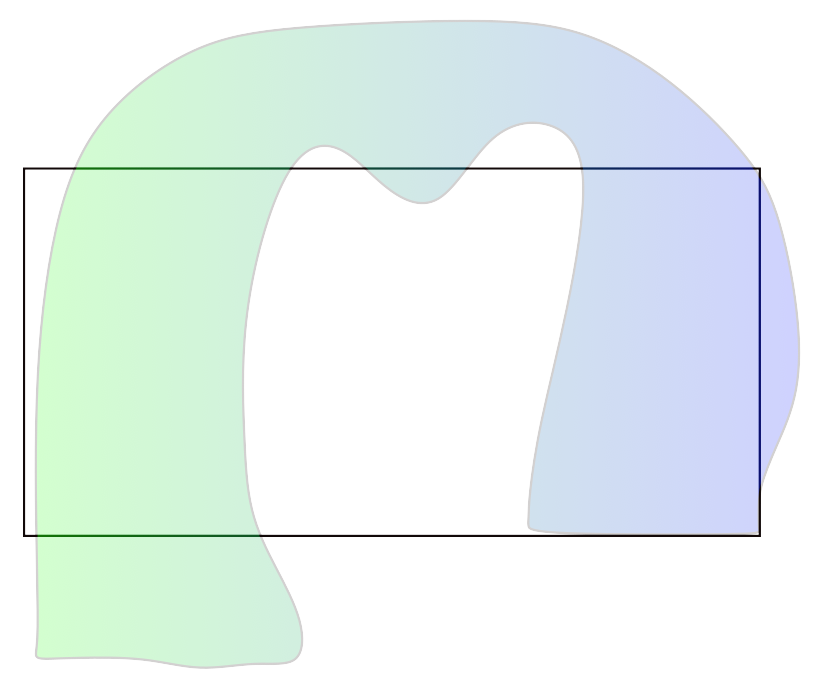}
\put(950,165){$P_{In}=B=g_p(B)$}
\put(650,100){$g_p(A)$}
\put(0,-50){$g_p(C)$}
\put(250,-45){$g_p(D)$}
\put(-40,600){$C$}
\put(950,600){$D$}
\put(-40,165){$A$}
\put(290,720){$g_p(ABCD)$}
\end{overpic}
\caption[Sketching $g_p$ as a rectangle map.]{\textit{Sketching $g_p$ as a rectangle map of $ABCD=H_p$.}}
\label{HOPFD1}
\end{figure}
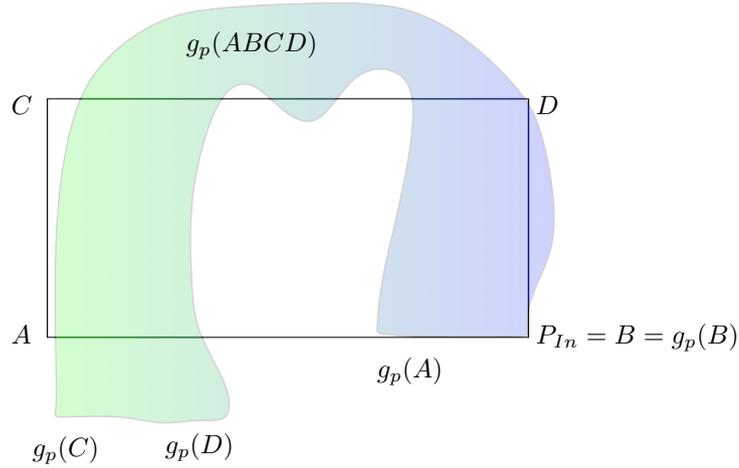

\begin{proof}
We prove Prop.\ref{doubla} directly, by applying a constriction similar to the original Smale Horseshoe argument (see \cite{S}). To begin, note that by Fig.\ref{HOPFD1} $ABCD\cap g^{-1}_{p}(ABCD)$  includes at least two components, $L_1,L_2$, s.t. both connect the $AC$ and $BD$ sides in non-trivial arcs (see the illustration in Fig.\ref{HOPFE}). Moreover, for $i,j\in\{1,2\}$ the set $g_{p}(L_i)$ connects the $AB$ and $CD$ sides, hence $g_{p}(L_i)\cap L_j\ne\emptyset$. Therefore, for $i=1,2$ we can split $L_i$ to at least three sets, $L_{i,1},L_{i,2}$ and $L_{i,3}$ (all of which connect the $AC$ and $BD$ sides) s.t. $g_p(L_{i,j})\subseteq L_j$ when $i,j\in\{1,2\}$, and $g_p(L_{i,3})\cap ABCD=\emptyset$ (see the illustration in Fig.\ref{HOPFF}). However, since $g^2_p(L_{i,j})$ also connects the $AB$ and $CD$ sides it follows we can also split $L_{i,j}$ to at least three closed sets, $L_{i,j,1}$, $L_{i,j,2}$ and $L_{i,j,3}$, all of which are closed topological discs which connect the $AC$ and $BD$ sides, s.t. $g^3_p(L_{i,j,k})\subseteq L_k$ when $k=1,2$, and $g^3_{p}(L_{i,j,3})\cap ABCD=\emptyset$.\\

 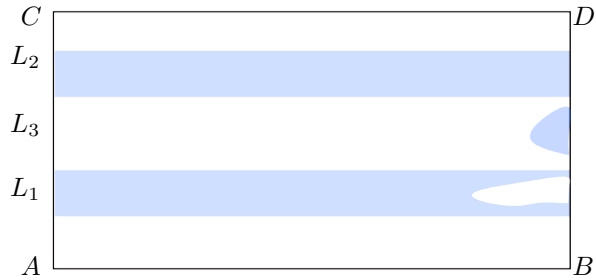
\begin{figure}[h]
\centering
\begin{overpic}[width=0.4\textwidth]{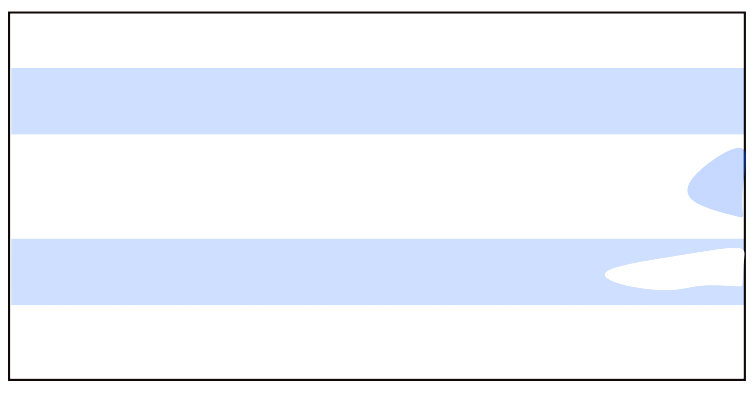}
\put(990,5){$B$}
\put(-50,480){$C$}
\put(990,480){$D$}
\put(-50,5){$A$}
\put(-70,155){$L_1$}
\put(-70,275){$L_3$}
\put(-70,405){$L_2$}
\end{overpic}
\caption[The sub-rectangles $L_1$ and $L_2$.]{\textit{The (topological) sub-rectangles $L_1$ and $L_2$. The other blue regions correspond to different components of $g^{-1}_p(ABCD)\cap ABCD$. $L_3$ denotes the region which is mapped outside under $g_p$. }}
\label{HOPFE}
\end{figure}

Repeating this argument ad infinitum, it follows that for any $\{i_0,...,i_n\}\in\{1,2\}^{n}$, $n>0$ we can define a closed topological disc $L_{i_0,...,i_n}$ s.t. $g^k_{p}(L_{i_0,...,i_n})\subseteq L_{i_k}$, $0\leq k\leq n$, while $g^{n+1}_{p}(L_{i_0,...,i_n})$ connects the $AB$ and $CD$ sides. Hence, given any $s\in\{1,2\}^\mathbf{N}$, $s=\{i_0,i_1,...\}$ by considering $L_s=\cap_{n\geq0} L_{i_0,...,i_n}$, we conclude $L_s$ is non-empty and forms the intersection of a nested sequence of closed topological discs in $ABCD$, all connecting the $AC$ and $BD$ sides - consequentially, $L_s$ also connects the $AC$ and $BD$ sides..\\

 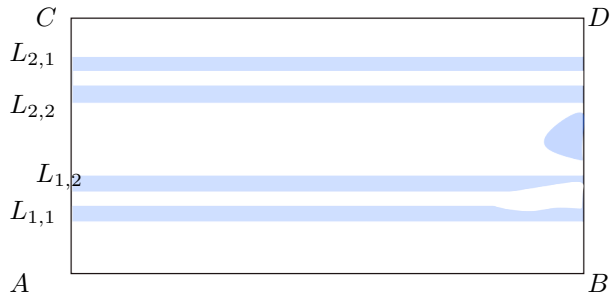
\begin{figure}[h]
\centering
\begin{overpic}[width=0.4\textwidth]{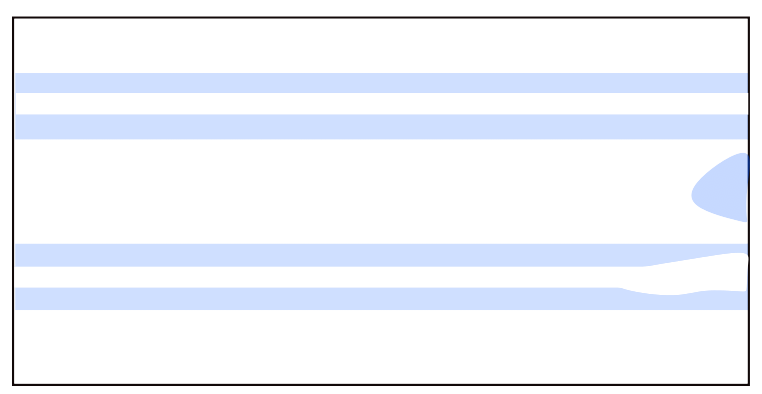}
\put(990,-25){$B$}
\put(-50,480){$C$}
\put(990,480){$D$}
\put(-100,-25){$A$}
\put(-50,185){$L_{1,2}$}
\put(-100,115){$L_{1,1}$}
\put(-100,315){$L_{2,2}$}
\put(-100,415){$L_{2,1}$}
\end{overpic}
\caption[Splitting $L_1$ and $L_2$.]{\textit{$L_1$ and $L_2$ are split to $L_{1,2},L_{1,1},L_{2,1},L_{2,2}$ s.t. $g_p(L_{i,j})\subseteq L_j$, and $g^2_p(L_{i,j})$ stretches from the $AB$ to the $CD$ side. }}
\label{HOPFF}
\end{figure}

Now, let us consider $L_s\cap g_{p}(ABCD)$ - it is immediate $g_{p}(L_i)\cap L_s\ne\emptyset$ for both $i=1,2$ - which implies we can partition $L_s$ to $L_{1,s}$ and $L_{2,s}$ s.t. $g^{-1}_{p}(L_{i,s})\subseteq L_i$, $i=1,2$ (see the illustration in Fig.\ref{HOPFG}). Similarly, by considering $g^n_p(ABCD)$ and iterating this argument, it follows that given any sequence $\{i_{-n},...,i_{-1}\}\in\{1,2\}^n$ we can generate a set $L_{i_{-n},...i_{-1},s}$ s.t. for every $-n\leq k\leq-1$ we have $g^{-k}_{p}(L_{i_{-n},...i_{-1},s})\subseteq L_{i_k}$. Therefore, repeating this proccess ad inifnitum, by considering $L_s\cap(\cap_{n\geq0} L_{i_{-n},...,i_{-1},s})$ we conclude that for every $\{...,i_{-1}\}=\omega$ there exists at least one component $L_{\omega,s}\subseteq L_s$ s.t. if $\{\omega,s\}=\{....,i_{-1},i_0,i_1,...\}$, for every $k\in\mathbf{Z}$ we have $g^k_{p}(L_{\omega,s})\subseteq L_{i_k}$.\\

 \begin{figure}[h]
\centering
\begin{overpic}[width=0.4\textwidth]{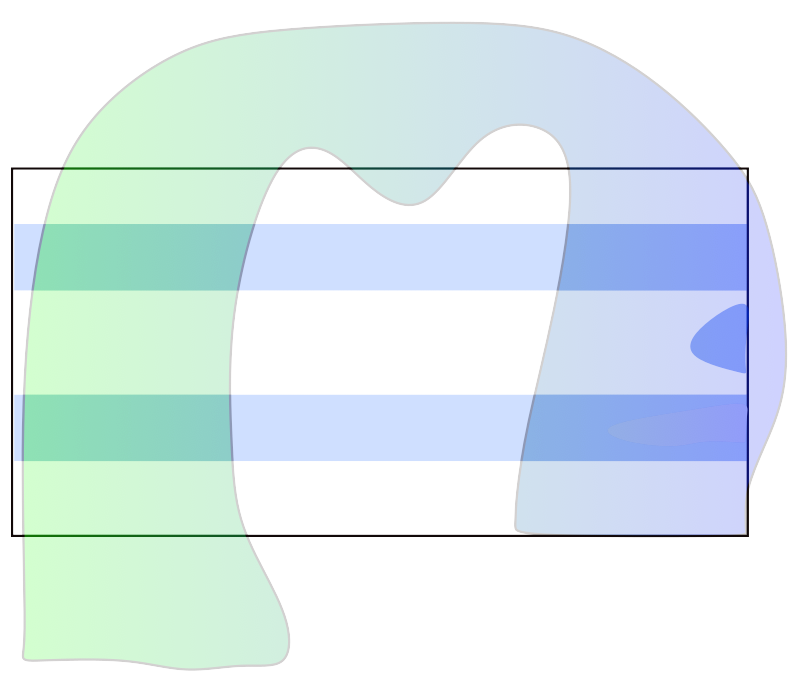}
\put(950,165){$B=g_p(B)$}
\put(650,100){$g_p(A)$}
\put(0,-50){$g_p(C)$}
\put(250,-45){$g_p(D)$}
\put(-60,600){$C$}
\put(950,600){$D$}
\put(-60,165){$A$}
\put(290,720){$g_p(ABCD)$}
\end{overpic}
\caption[Decomposing $L_1$ and $L_2$.]{ \textit{$g^{-1}_p(L_i), i=1,2$ intersects both $L_1$ and $L_2$}.}
\label{HOPFG}
\end{figure}
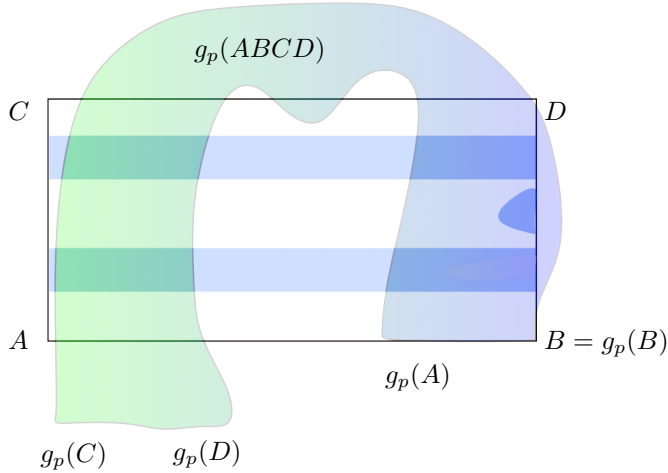

Consequentially, it again follows that $L_{\omega,s}$ in homeomorphic to the intersection of a nested sequence of closed topological disc - as such, by deforming $g_p$ isotopically (if necessary) we can ensure $L_{\omega,s}$ is a convex set, hence contractible. All in all, for every $\mu\in\{1,2\}^\mathbf{Z}$, $L_\mu$ is well defined - and when $\mu$ is periodic of minimal period $k$,  as $g^k_p(L_\mu)\subseteq L_\mu$, by the Brouwer Fixed-Point Theorem it follows $L_\mu$ includes at least one periodic point of minimal period $k$. Now, set $I$ as the collection of $L_\mu$, $\mu\in\{1,2\}^\mathbf{Z}$, and set $\xi:I\to\{1,2\}^\mathbf{Z}$ by $\xi(x)=\mu$ for $x\in L_\mu$ - by the discussion above, it is immediate $\xi$ is onto. It is easy to see that since $g_{p}$ is continuous on $\overline{H_p}=ABCD$, it is also continuous on $I$ - which makes $\xi$ continuous as well. Moreover, it is also immediate that we have $\xi\circ g_{p}=\zeta\circ\xi$, where $\zeta:\{1,2\}^\mathbf{Z}\to\{1,2\}^\mathbf{Z}$ is the double-sided shift - and since by construction for every $\mu\in\{1,2\}^\mathbf{Z}$ we have $L_\mu=\xi^{-1}(\mu)$, we conclude $\xi^{-1}(\mu)$ is contractible.\\

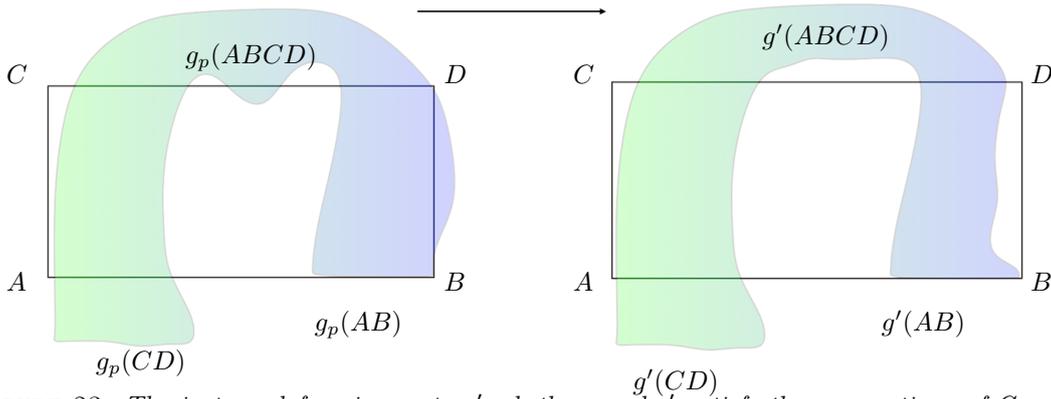
\begin{figure}[h]
\centering
\begin{overpic}[width=0.75\textwidth]{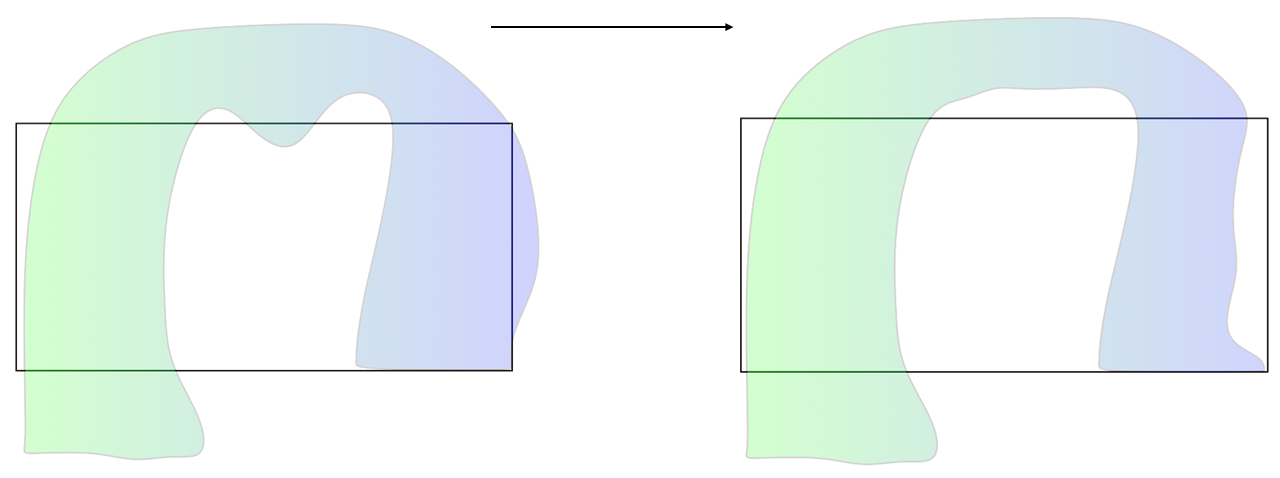}%
\put(1000,70){$B$}
\put(1000,280){$D$}
\put(410,280){$D$}
\put(410,70){$B$}
\put(-30,280){$C$}
\put(-30,70){$A$}
\put(150,300){$g_p(ABCD)$}
\put(730,320){$g'(ABCD)$}
\put(60,-10){$g_p(CD)$}
\put(850,30){$g'(AB)$}
\put(600,-30){$g'(CD)$}
\put(540,70){$A$}
\put(540,280){$C$}
\put(280,30){$g_p(AB)$}
\end{overpic}
\caption[Isotoping $g_p$ to $g'$.]{\textit{The isotopy deforming $g_p$ to $g'$ - both $g_p$ and $g'$ satisfy the assumptions of Cor.\ref{deformation11}. Moreover, this isotopy removes every component in $J\setminus I$ (as defined below).} }\label{HOPFHH}
\end{figure}

To conclude the proof of The proof of Prop.\ref{doubla} it remains to show every periodic orbit in $Q$ for $f_p$ is continuously deformed to a periodic orbit in $I$. By the discussion preceding Cor.\ref{corhalf}, we know every periodic orbit in $Q$ is continuously deformed to a periodic orbit for $g_p$ in $H_p$. Now, let us note that if $J$ is the invariant set of $g_p$ in $ABCD$, we can always smoothly deform the vector field $G_p$ as to induce an isotopy of $g_p:\overline{H_p}\setminus\delta\to \overline{D_\alpha}$ to $g':\overline{H_p}\setminus\delta\to\overline{D_\alpha}$ which removes every component of $J\setminus I$ - moreover, we can smoothly deform $G_p$ s.t. the isotopy keeps $g_p(AB)$ and $g_p(CD)$ fixed (see the illustration in Fig.\ref{HOPFHH}).\\

Since $g_p$ is isotopic to $g'$ on $\overline{H_p}\setminus\delta$, $g'$ must also be isotopic to $f_p$ in $\overline{H_p}\setminus\delta$ - and in particular, $g'$ satisfies the assumptions of Cor.\ref{deformation11}. Therefore, by Cor.\ref{deformation11} we conclude the periodic orbits in $Q$ persists under the isotopy from $f_p$ to $g'$ - and in particular, no minimal period is changed, nor do any two periodic orbits in different components of $Q$ collapse into one another. It therefore follows that as we isotopically deform $f_p$ to $g_p$ the periodic orbits in $Q$ are continuously deformed to a periodic orbit in $I$ (of the same minimal period), and the assertion follows.
\end{proof}
\begin{remark}
    Note that given any curve $\gamma\subseteq ABCD$ connecting the $AB$ and $BD$ sides, $g_p(\gamma)\cap ABCD$ includes at least two components, s.t. each connects the $AB$ and $CD$ sides. This implies $g_p:ABCD\to\overline{D_\alpha}$ is a topological horseshoe (see \cite{KY}) - in \cite{XSYS03} it was proven the first-return map for the Rossler system is also a topological horseshoe (albeit in a non-heteroclinic setting).
\end{remark}
\subsection{Stage $III$ - deforming $G_p$ to $G$ and concluding the proof.}
Having constructed the vector field $G_p$, we now smoothly deform it in $S^3$ to $G$, a vector field whose first-return map is a Smale Horseshoe map, thus concluding the proof of Th.\ref{deform}.\\

To begin, denote by $J$ denote the maximal invariant set of $g_{p}$ in the rectangle $ABCD$, where $I$ is the invariant set given by Prop.\ref{doubla} (by definition, $I\subseteq J$). Now, smoothly deform in the three-sphere $\mathbf{S}^3$ the vector field $G_p$ to a vector field $G$, by isotoping $g_p$ to a first-return map $g:\overline{U_p}\to\overline{U_p}$ as indicated below:

\begin{itemize}
    \item First, we begin by isotopically removing from $J$ all the components of $J\setminus I$, as illustrated in Fig.\ref{HOPFHH} and Fig.\ref{HOPFI}. We can do so since $J\subseteq\overline{H_p}$, and $g_p:\overline{H_p}\to \overline{D_\alpha}$ is continuous.  
    \item Second, for every initial condition in $s\in\overline{U_p}\setminus I$, we move its forward trajectory s.t. it eventually hits $D_{In}$ transversely - that is, we ensure every trajectory which lies away from $I$ eventually flows towards the sink $P'_{In}$ (see the illustration in Fig.\ref{HOPFI}).
    \item Finally, let $W$ denote the collection of forward and backwards trajectories for initial conditions in $I$ - to complete the deformation, collapse every component in $W$ to a singleton (we can do so, since by Prop.\ref{doubla} every component of $I$ is contractible).  
\end{itemize}

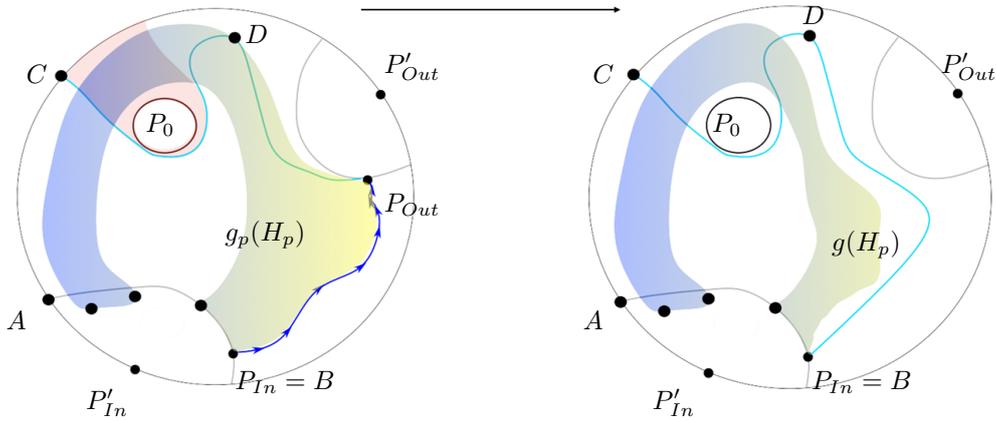
\begin{figure}[h]
\centering
\begin{overpic}[width=0.75\textwidth]{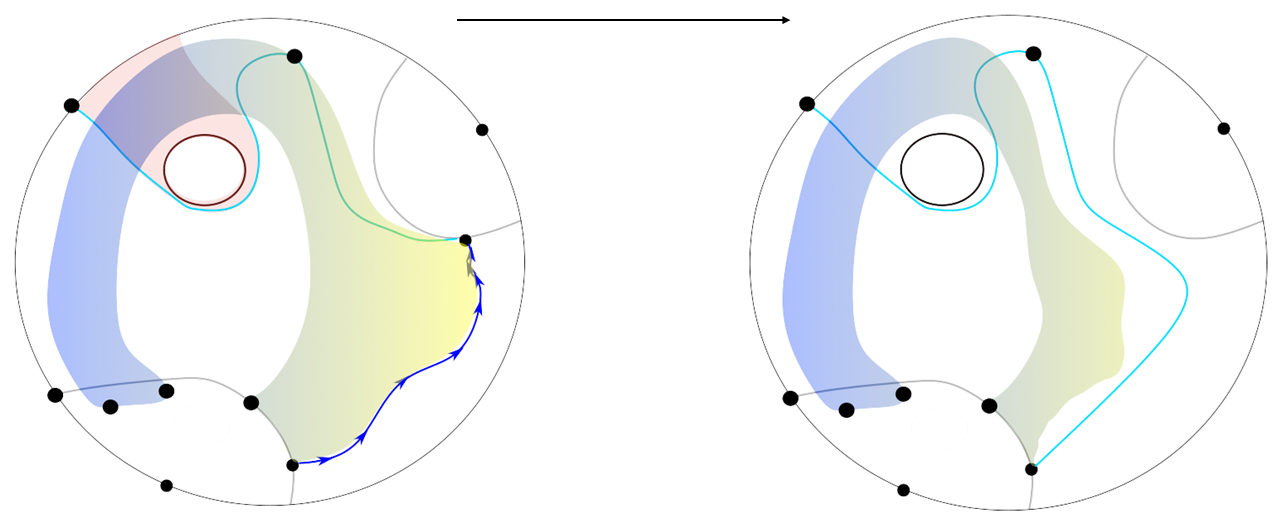}%
\put(940,330){$P'_{Out}$}
\put(380,330){$P'_{Out}$}
\put(140,270){$P_0$}
\put(80,-10){$P'_{In}$}
\put(810,10){$P_{In}=B$}
\put(230,10){$P_{In}=B$}
\put(0,70){$A$}
\put(580,70){$A$}
\put(20,320){$C$}
\put(590,320){$C$}
\put(240,360){$D$}
\put(800,380){$D$}
\put(650,-10){$P'_{In}$}
\put(380,190){$P_{Out}$}
\put(710,270){$P_0$}
\put(220,160){$g_p(H_p)$}
\put(830,150){$g(H_p)$}
\end{overpic}
\caption[Isotoping $g_p$ to $g$.]{\textit{The isotopy deforming $g_p$ to $g$, a Smale Horseshoe map.} }\label{HOPFI}
\end{figure}

By its definition, the deformation above induces an isotopy of $g_p:ABCD\to\overline{U_p}$ to some diffeomorphism $g:ABCD\to\overline{U_p}$ - the first-return map for $G$, as illustrated in Fig.\ref{HOPFI}. Setting $V$ as the invariant set of $g$ in $ABCD$, by the construction above $V$ is generated from $I$ by collapsing every component of $I$ to a singleton. As proven in Prop.\ref{doubla}, there exists a continuous surjection $\xi:I\to\{1,2\}^\mathbf{Z}$ $\xi\circ g_p=\zeta\circ\xi$ (where $\zeta:\{1,2\}^\mathbf{Z}\to\{1,2\}^\mathbf{Z}$ is the double-sided shift) - which implies there exists a homeomorphism $\mu:V\to\{1,2\}^\mathbf{Z}$ s.t. $\mu\circ g=\zeta\circ\mu$. Or, in other words, $\mu\circ g\circ\mu^{-1}=\zeta$, i.e., the dynamics of $g$ on $V$ are conjugate to those of the double-sided shift.\\

As such, it now follows $g:ABCD\to\overline{U_p}$ is a Smale Horseshoe map - hence, it is hyperbolic on $V$, i.e., there exists some $1>\lambda>0$ and $\theta>1$ s.t. for every $s\in V$ its differential w.r.t. $g$, $D_sg$, has eigenvalues $\gamma$ and $\theta$. Summarizing our results, we conclude:

\begin{corollary}
\label{gex}    The first-return map $g:V\to V$ is conjugate to a Smale Horseshoe map on its invariant set. In particular, $g$ is hyperbolic on $V$.
\end{corollary}

Therefore, to conclude the proof of Th.\ref{deform} it remains to prove every knot type generated by $G$ as a periodic trajectory (save perhaps for two) is also generated by $F_p$ as a periodic trajectory. To do so, let us note the maps $f_p:\overline{H_p}\setminus\delta\to\overline{U_p}$ and $g:\overline{H_p}\setminus\delta\to\overline{U_p}$ are isotopic. It is also easy to see that by construction, the periodic orbits for $f_p$ in $Q$ given by Th.\ref{th31} are all continuously and isotopically deformed to periodic orbits for $g$ in $V$ - of the same minimal period, and moreover, without collapsing into one another.\\

As such, it follows that if $T_1$ is a solution curve generated as a periodic trajectory for the vector field $G$, as $G$ is smoothly deformed to $F_p$, $T_1$ is deformed by an isotopy of $\mathbf{R}^3$ to $T_0$ - where $T_0$ is a solution curve corresponding to a periodic trajectory for $F_p$. Moreover, if $T_{1,1}$ and $T_{1,2}$ are two distinct periodic trajectories for $G$, they do not collapse to one another under this deformation: i.e., when we smoothly deform $G$ back to $F_p$, $T_{1,2}$ and $T_{1,1}$ are isotopically deformed to $T_{0,2}$ and $T_{0,1}$, two distinct periodic trajectories for $F_p$.\\

Finally, recall $F_p$ has precisely two fixed points in $\mathbf{R}^3$, and that $F_p$ is deformed to $G$ by opening these fixed points to periodic trajectories by Hopf-bifurcations. As such, there exist at most two periodic trajectories for $G$ which are closed to fixed-points as we return to $F_p$. Therefore, we conclude:

\begin{corollary}
Every periodic trajectory for $G$ (save perhaps for two) can be deformed to a periodic trajectory of $F_p$ by an isotopy of $\mathbf{R}^3$. Consequentially, save possibly for two knot-types, every knot type generated as a periodic trajectory for $G$ is also realized as a periodic trajectory for $F_p$.
\end{corollary}
The proof of Th.\ref{deform} is now complete.
\end{proof}

\section{Classifying the knot types and proving their persistence}
Having proven Th.\ref{deform}, in this section we study its corollaries. Our main results in this Section are Th.\ref{geometric} and Th.\ref{persistence}. Together, these two results classify the knot types generated as periodic trajectories for the Rössler system at trefoil parameters - as well as prove their persistence under. perturbations of the flow. As will be made clear below, this topological classification of periodic trajectories would essentially allow us to consider the dynamics of $G$, the vector field given by Th.\ref{deform}, as a topological lower bound for the dynamics of $F_p$ - i.e., the dynamics of $F_p$ are complex at least like those of $G$.\\

We begin by introducing the following definition:
\begin{definition}
\label{model} Let $\phi$ be a $C^\infty$ flow on a $3$-manifold $M$, with a hyperbolic non-wandering set, $K\subseteq M$. A \textit{\textbf{plug}} for $\phi$ is a $3$-manifold $N$ with boundary, satisfying the following properties:
\begin{itemize}
    \item There exists a flow $\psi$ on $N$, transverse to $\partial N$ and hyperbolic on its invariant set in $N$.
    \item  $\phi,\psi$ are orbitally equivalent on their respective non-wandering sets in $M,N$.
\end{itemize}
A \textbf{\textit{model}} is a plug with a minimal total genus.
\end{definition}
As far as models are concerned, we have the following theorem, proven in \cite{BB}:

\begin{claim}\label{begbon}
Given $\phi,M,K$ as above, there exists a model - and moreover, the model is unique up to orbital equivalence of the flows.
\end{claim}

Therefore, by applying Th.\ref{begbon}, we now prove the following fact, which is a corollary of both Th.\ref{deform} and Th.\ref{begbon}:
\begin{proposition}\label{orbiteq}
Let $p\in P$ be a trefoil parameter for the Rössler system, and let $G$ be the vector field given by Th.\ref{deform}. Then, we have the following:
\begin{itemize}
    \item The flow generated by $G$ is unique - up to orbital equivalence on the invariant set of its first-return map $g$ in $ABCD$.
    \item There exists a Template $T$, independent of $G$, s.t. any knot type on $T$ is generated by $G$ - and vice versa.
    \item Every knot type on $T$ (save perhaps for two) is realized as a periodic trajectory for $F_p$ - that is, one can associate a template with the Rössler system at trefoil parameters.
\end{itemize}
\end{proposition}
\begin{proof}
By Th.\ref{deform}, $G$ generates a smooth flow in $S^3$ - and by Cor.\ref{gex}, that flow is hyperbolic on an invariant set. Also recall that by construction the vector field $G$ has precisely two fixed points in $S^3$, $P'_{In}$, a sink, and $P'_{Out}$, a source (both generated by opening the fixed-points by Hopf-bifurcations).\\

Let $B_{In},B_{Out}$ denote open balls centered at $P'_{In},P'_{Out}$ (respectively) s.t. $G$ is transverse to the respective boundaries of both, $S_{In},S_{Out}$ - in particular, on $S_{In}$ $G$ points into $B_{In}$ while on $S_{Out}$ $G$ points outside $B_{Out}$. Consequentially, $B_{In},B_{Out}$ both lie in some positive distance from the periodic trajectories for $G$, which implies all the periodic trajectories for $G$ lie in $S^3\setminus(B_{In}\cup B_{Out})=M$. $M$ has a total genus $0$ - therefore, since by Th.\ref{deform} $G$ is hyperbolic on its maximal invariant set in $M$, by Def.\ref{model} it follows the flow generated by $G$ on $M$ is a model flow. Therefore, from Th.\ref{begbon}, we know the flow generated by $G$ on $M$ is unique up to orbital equivalence.\\

To continue, let $T$ denote the template corresponding to $G$ in $M$ - as given by The Birman-Williams Theorem (see Th.\ref{BIRW}). Additionally, let $p\in P$ denote a trefoil parameter for the Rössler system, and recall we denote the corresponding vector field by $F_p$ (see Eq.\ref{Field}). By the discussion above, it follows that when we deform the vector field $F_p$ to $G$ per Th.\ref{deform}, no matter how we generate $G$, it would always be a model - and by Th.\ref{begbon} these models are all orbitally equivalent to one another. It therefore follows the Template $T$ is also unique, and independent of the vector field $G$.\\

Finally, recall that by Th.\ref{deform}, every periodic trajectory for $G$ (save possibly for two) can be deformed to a periodic trajectory for $F_p$ by an isotopy of $\mathbf{R}^3$ - therefore, given any knot type $T'$ encoded by $T$ (save perhaps for two), there exists a periodic trajectory $T_0$ for $F_p$ whose type is also $T'$. All in all, the proof Prop.\ref{orbiteq} is now complete.
\end{proof}

\begin{figure}[h]
\centering
\begin{overpic}[width=0.4\textwidth]{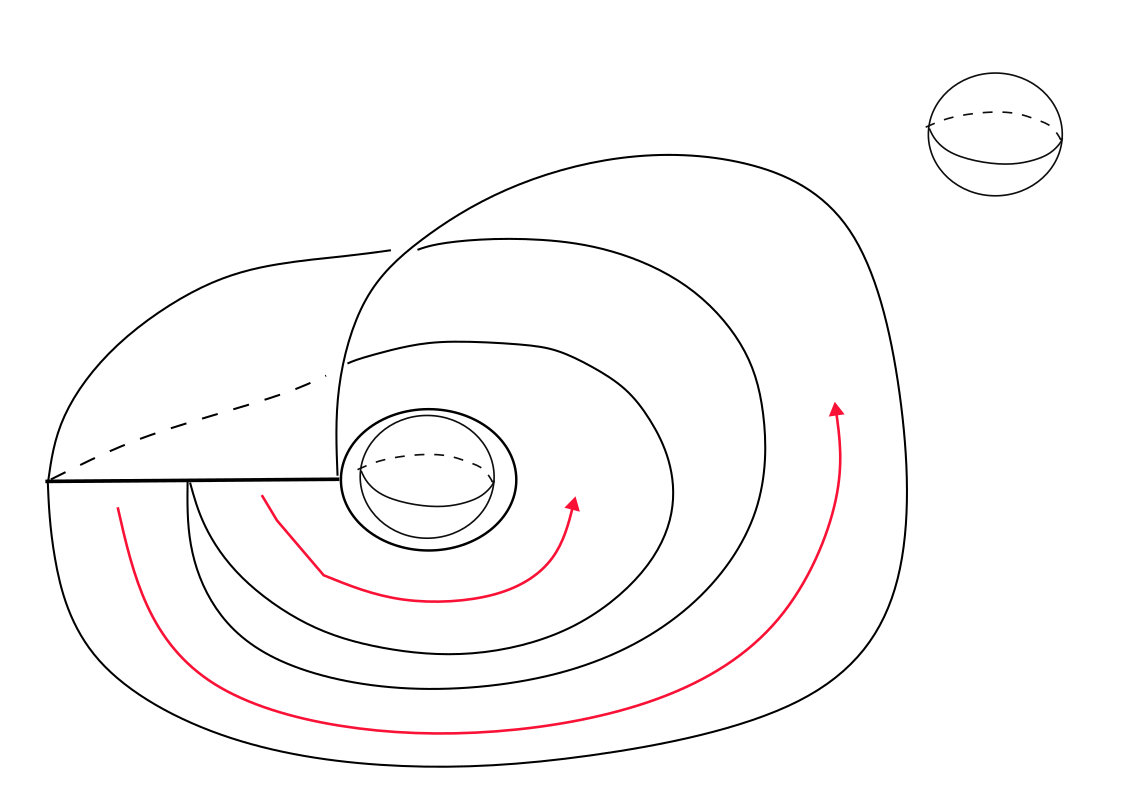}
\put(-30,260){$B$}
\put(350,270){$B_2$}
\put(860,570){$B_1$}
\end{overpic}
\caption[A horseshoe template.]{\textit{A horseshoe template, $V$ suspended around the trefoil.}}
\label{fig34}
\end{figure}

Having proven Prop.\ref{orbiteq}, we now apply Th.\ref{begbon} yet again - this time, to explicitly characterize the template $T$ given by Prop.\ref{orbiteq}. To do so, let us first consider a smooth vector field $H$ on $S^3$ which is generated by suspending the horseshoe map given in Fig.\ref{fig35}. In addition, assume $H$ has two fixed points in $S^3$ - first, $p_1$, a source, and second, $p_2$, a sink. Now, let $S_1$ denotes a sphere, enclosing a ball $B_1$ s.t. $p_1\in B_1$, on which $H$ points outside of $B_1$. Conversely, let $S_2$ denote another sphere, enclosing a ball $B_2$ s.t. $p_2\in B_2$, on which $H$ points into $B_2$. Now, suspend the horseshoe in accordance with the motion described in Fig.\ref{fig34}. Let us remark that since Smale Horseshoes are hyperbolic on their invariant set, by Th.\ref{BIRW} the periodic trajectories of $H$ correspond to a template $V$ as in Fig.\ref{fig34}. Moreover, by Th.\ref{begbon}, the vector field $H$ generates a model flow on $S^3\setminus(B_1\cup B_2)$. \\

\begin{figure}[h]
\centering
\begin{overpic}[width=0.45\textwidth]{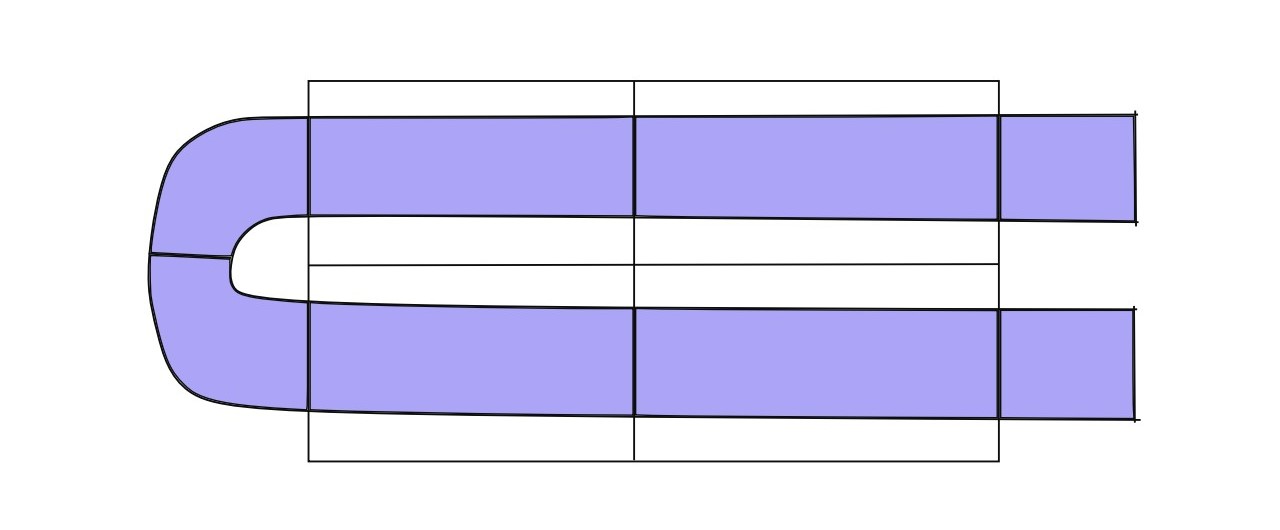}
\put(900,300){$h(b)$}
\put(900,240){$h(a)$}
\put(900,160){$h(c)$}
\put(900,100){$h(d)$}
\put(480,0){$F$}
\put(220,30){$c$}
\put(0,200){$h(F)$}
\put(220,360){$d$}
\put(780,30){$a$}
\put(780,360){$b$}
\put(190,200){$B$}
\put(590,360){$h(S)$}
\end{overpic}
\caption[Suspending a horseshoe and creating a template.]{\textit{$h:S\to S$, a Smale Horseshoe map. $F$ in Fig.\ref{fig35} denotes the fold-line of $h$ in $S$, and $B$ the branch line on the template $V$.}}
\label{fig35}
\end{figure}
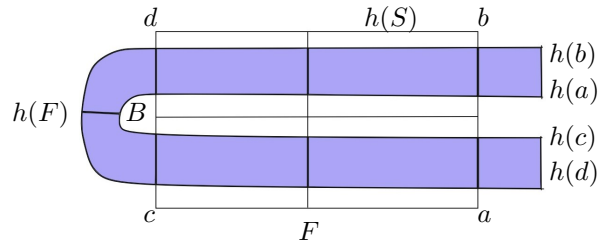

Now, recall the first-return map generated by $G$, $g:ABCD\rightarrow \overline{U_p}$, with $ABCD$ as in the proof of Th.\ref{deform}, depicted in Fig.\ref{fig36}. Comparing $h$ with $g$, we see the vector fields $H,G$ suspend a horseshoe in the exact same way, and more importantly - both are models for the same flow (see the discussion before Th.\ref{begbon}). Hence by Th.\ref{begbon} we conclude the flows generated by $G$ and $H$ are orbitally equivalent. It now follows $T$, the template given by Prop.\ref{orbiteq} is the same as the template $V$ - therefore we conclude:
\begin{corollary}
    \label{knotype}
Let $G$ be as in Th.\ref{deform}, let $T$ be the template from Prop.\ref{orbiteq}, and let $V$ denote the template defined above. Then, $V=T$, and in particular, the template $V$ describes all the knot-types generated by the vector field $G$.
\end{corollary}

\begin{figure}[h]
\centering
\begin{overpic}[width=0.45\textwidth]{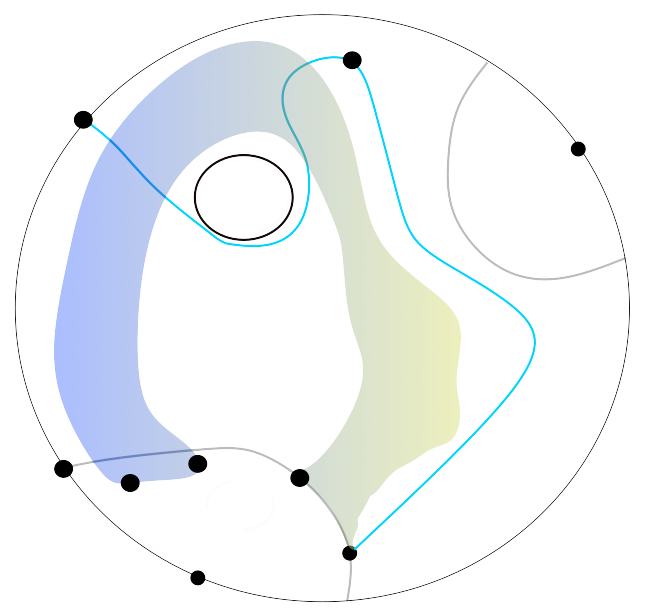}
\put(520,25){$B$}
\put(300,-10){$P'_{In}$}
\put(200,135){$g(CD)$}
\put(0,210){$A$}
\put(535,880){$D$}
\put(95,780){$C$}
\put(400,130){$g(AB)$}
\put(750,750){$P'_{Out}$}
\put(590,340){$g(ABCD)$}
\end{overpic}
\caption[$g$ as a Smale Horseshoe map.]{\textit{The horseshoe map $g:ABCD\to \overline{U_p}$.}}
\label{fig36}
\end{figure}
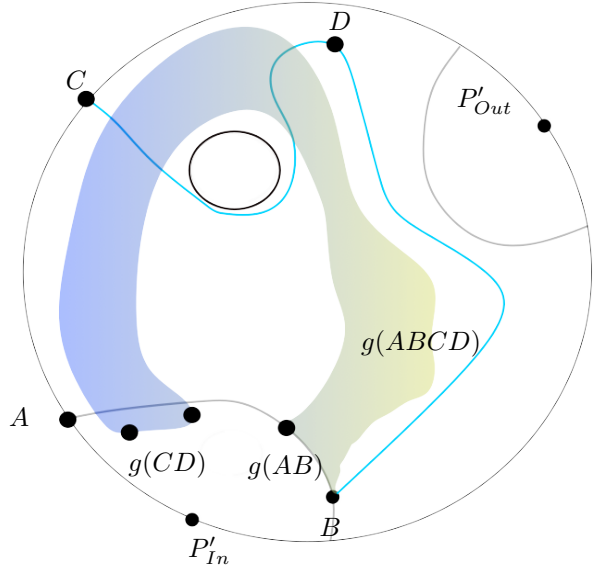

However, we can say more. Consider the template $T=V$ sketched on the left in Fig.\ref{fig37} below. By deforming that template isotopically (in $\mathbf{R}^3$), $T$ is isotoped to the $L(0,1)$ template, termed the \textbf{Lorenz $0-1$ Template},  sketched on the right of Fig.\ref{fig37} (see \cite{HW} for more details on $L(0,1)$). In other words, every knot type encoded $T$ is also encoded $L(0,1)$ and vice versa - hence, by Cor.\ref{knotype} we immediately conclude:

\begin{corollary}\label{corlorenz}
Let $G$ be the vector field given by Th.\ref{deform}. Then, every knot type on $L(0,1)$ is generated by $G$ as a periodic trajectory - conversely, every periodic trajectory for $G$ corresponds to precisely one knot type on the template $L(0,1)$.
\end{corollary}

\begin{figure}[h]
\centering
\begin{overpic}[width=0.45\textwidth]{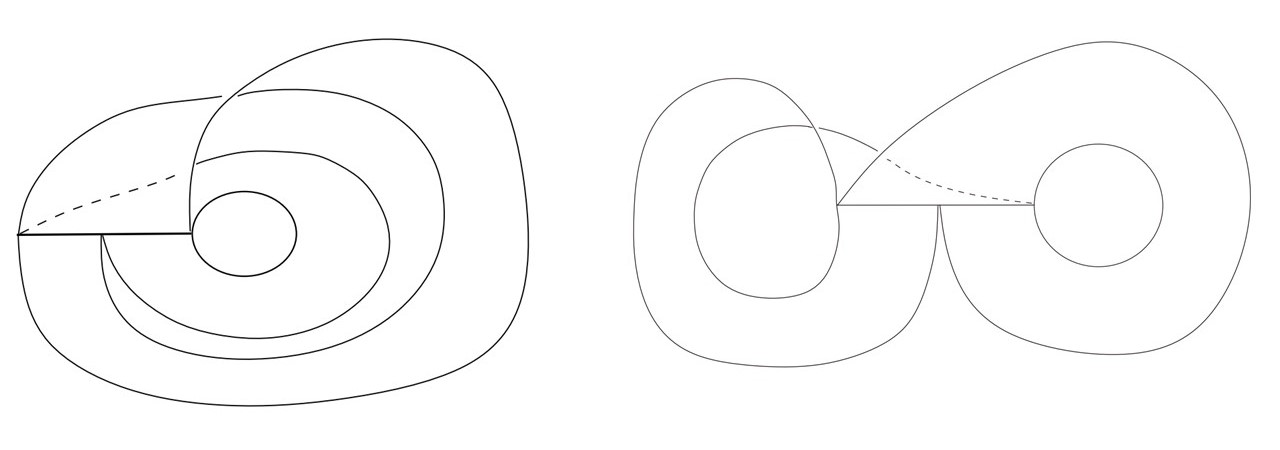}
\put(650,320){$L(0,1)$}
\put(150,320){$T$}
\end{overpic}
\caption[$T$ and $L(0,1)$.]{\textit{The two templates, $T$ and $L(0,1)$. $T$ can be deformed to $L(0,1)$ by an isotopy of $\mathbf{R}^3$.}}
\label{fig37}
\end{figure}

Having proven the $L(0,1)=T$ template describes the dynamical complexity generated by $G$, we are now ready to prove Th.\ref{geometric} - with which we conclude this section. To do so, recall that given a trefoil parameter $p\in P$, we denote by $F_p$ the corresponding vector field (see Eq.\ref{Field}) - that is, $\dot{x}=F_p(x)$ is the Rössler system corresponding to the parameter values $p=(a,b,c)$). We now prove:

\begin{theorem}\label{geometric}
Let $p\in P$ be trefoil parameter for the Rössler system. Then, save perhaps for two knot types, $F_p$ generates every knot-type present in the $L(0,1)$ template. In particular, $F_p$ generates infinitely many torus knots (as defined in Def.\ref{torusknot}).
\end{theorem}
\begin{proof}
By Th.\ref{deform}, every periodic trajectory for $G$ (save perhaps for two) can be smoothly deformed to a periodic trajectory for $F_p$ by some isotopy of $\mathbf{R}^3$. Since by Cor.\ref{corlorenz} the knots types corresponding to the periodic trajectories of $G$ are precisely those on the $L(0,1)$ template, it follows $F_p$ also generates every knot-type in $L(0,1)$ (save perhaps for two). Finally, since $L(0,1)$ includes infinitely many Torus knots (see Th.$6.1.2.b$ in \cite{HW}), Th.\ref{geometric} follows.
\end{proof}
Before concluding this section, we now prove Th.\ref{persistence}, which shows the periodic dynamics given by Th.\ref{geometric} all persist in the parameter space of the Rössler system. Namely, we prove the following result:

\begin{theorem}
  \label{persistence}  Let $p\in P$ be a trefoil parameter, and let $v\in P$ be another parameter, s.t. $v\ne p$. Then, given any knot type $T$ on $L(0,1)$ (save perhaps for one) there exists an $\epsilon>0$ s.t. whenever $||v-p||<\epsilon$, $F_v$ generates a periodic trajectory $T_v$ whose knot-type is $T$.
\end{theorem}
\begin{proof}
        Recall the set $Q$, given by Th.\ref{th31}. By Th.\ref{geometric} and Th.\ref{deform}, we know every periodic trajectory $T_G$ for the vector field $G$ (save perhaps for one) is deformed by an isotopy of $\mathbf{R}^3$ to a periodic trajectory $T_p$ for $F_p$. Moreover, recalling the cross-section $U_p$ and the proof of Th.\ref{deform} (see the discussion before Th.\ref{th21}), we know $T_p\cap U_p\subseteq Q$ - that is, $T_p$ corresponds to some periodic orbit for the first-return map $f_p$ in $Q$ (see the discussion before Th.\ref{th31} and Th.\ref{th21}). Moreover, let us recall the by Th.\ref{deform} $T_G,T_p$ have the same knot type - therefore, since by Th.\ref{geometric} every periodic trajectory $T_G$ lies on the template $L(0,1)$, to prove Cor.\ref{persistence} it would suffice to prove the knot type of $T_p$ persists under sufficiently small perturbations in $P$.\\
        
To continue, following Ch.VII.5 of \cite{Dold}, we define the Fixed Point Index as follows: let $V\subseteq \mathbf{R}^2$ be a topological disc, and let $f:V\to\mathbf{R}^2$ be continuous. Provided $Per(f)=\{x\in V|f(x)=x\}$ is compact in $V$, we define the \textbf{Fixed-Point Index} of $f$ in $V$ to be the degree of $f(x)-x$ in $V$ - in particular, when the Fixed-Point Index is non-zero $f$ has a fixed point in $V$. Now, let $f_t:V\to\mathbf{R}^2$, $t\in[0,1]$ be a homotopy of continuous maps, and set $Per=\{(x,t)|t\in[0,1],f_t(x)=x\}$ - provided $Per$ is compact in $V\times[0,1]$, the Fixed-Point Index is constant along the homotopy: that is, under these assumptions $f_0,f_1$ have the same Fixed Point Index (for a proof, see Th.VII.5.8 in \cite{Dold}).\\

\begin{figure}[h]
\centering
\begin{overpic}[width=0.5\textwidth]{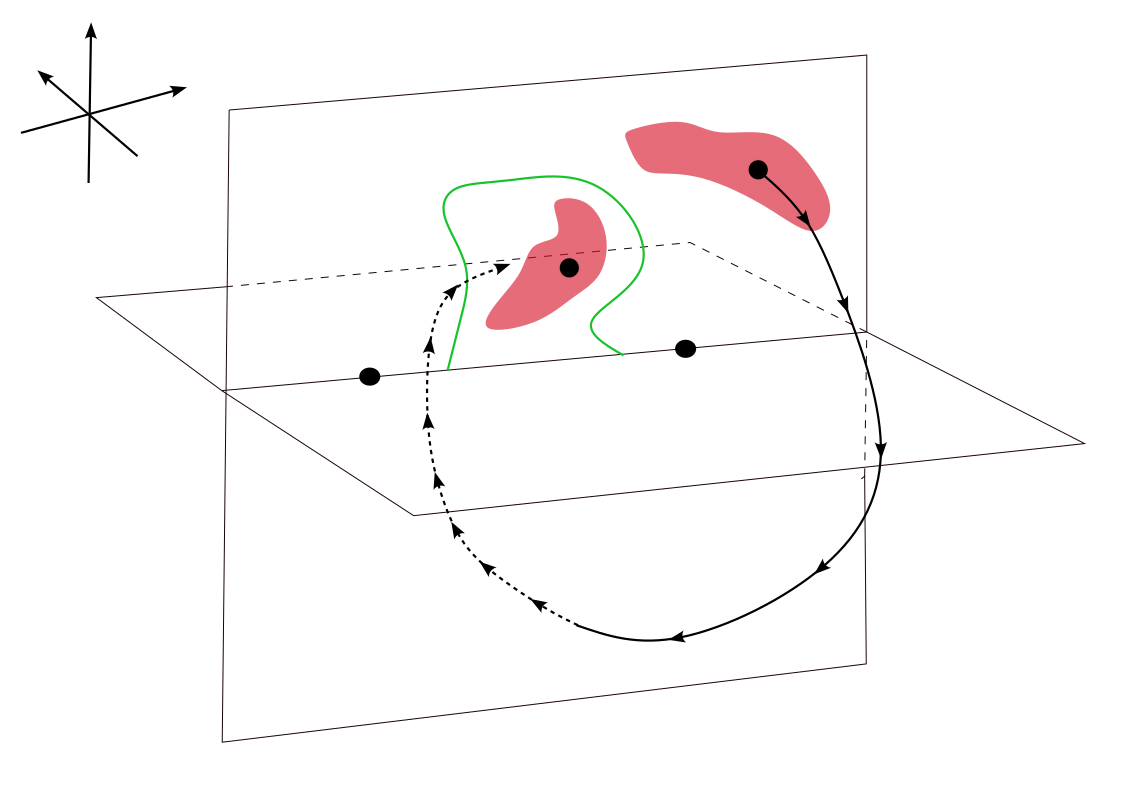}
\put(70,690){$z$}
\put(320,550){$D_1$}
\put(405,500){$D_2$}
\put(450,420){$f_p(B_p)$}
\put(600,480){$\rho$}
\put(600,550){$B_p$}
\put(800,480){$U_p$}
\put(170,630){$x$}
\put(10,650){$y$}
\put(455,60){$\{\dot{y}=0\}$}
\put(455,300){$\{\dot{x}=0\}$}
\put(270,350){$P_{In}$}
\put(630,380){$P_{Out}$}
\end{overpic}
\caption[Fig31]{The flow lines connecting $B_p,f_p(B_p),...,f^{k}_p(B_p)$ are transverse to $U_p$. Consequentially, $f_p,...,f^k_p$ are continuous on $B_p$. See the Prop.4.2 in \cite{I} for a proof.}
\label{pers1}
\end{figure}

We now apply these notions as follows. To do so, let $x_p$ be a point in $T_p\cap U_p$ - that is, $x_p$ is periodic for the first-return map $f_p:U_p\to U_p$ (wherever defined), of, say, minimal period $k$. Now, let us recall that by Lemma 4.5 in \cite{I}, there exists a topological disc $B_p\subseteq U_p$ satisfying the following (in Lemma 4.5 in $I$, $B_p$ is denoted by $V_s$):
\begin{itemize}
    \item $x_p\in B_p$.
    \item $f_p,...,f_p^k$ are all continuous on $B_p$ - that is, the flow lines connecting $f^i_p(B_p),f^{i+1}_p(B_p)$, $0\leq i\leq k$ are all transverse to $U_p$ (see the illustration in Fig.\ref{pers1}). 
    \item For every $i\ne j$, $0\leq i,j<k$, $f^i_p(B_p)\cap f^j_p(B_p)=\emptyset$ (see the illustration in Fig.\ref{pers1}).
    \item The set $Per(f_p)=\{x\in B_p|f^k_p(x)=x\}$ is compact in $B_p$ (that is, it lies away from $\partial B_p$) - moreover, the Fixed-Point Index of $f^k_p$ on $B_p$ is non-zero.
\end{itemize}

Now, let us perturb the vector field $F_p$ to some $F_v$, $v\in P$ - recall this perturbation deforms the cross-section $U_p$ continuously to $U_v$ (see Cor.\ref{TR}). By the discussion above, it also follows $B_p$ to continuously deformed to $B_v$, some topological disc on $U_v$ (see the discussion immediately before Lemma \ref{obs}). Recalling we denote by $f_v:\overline{U_v}\to\overline{U_v}$ the first-return map of $U_v$ (wherever defined), by the discussion above we conclude that when $v$ is sufficiently close to $p$, we have the following:
\begin{itemize}
    \item The flow lines connecting $f^i_v(B_v),f^{i+1}_p(B_v)$, $0\leq i\leq k$ are all transverse to $U_v$ - that is, $f_v,...,f^k_v$ are all continuous on $B_v$.
    \item For every $i\ne j$, $0\leq i,j<k$, $f^i_v(B_v)\cap f^j_v(B_v)=\emptyset$.
    \item For every $0<i\leq k$, $f^i_p:B_p\to U_p$ and $f^i_v:B_v\to U_v$ are homotopic. 
\end{itemize}

Consequentially, $f_p,...,f^k_p$ are continuously deformed to $f_v,...,f^k_v$  - i.e., there exists a homotopy $f_t$ of disc maps, $t\in[0,1]$, s.t. $f_0=f^k_p$ and $f_1=f^k_v$. Moreover, because $f^k_p$ has no fixed points in $\partial B_p$, it follows that whenever $v$ is sufficiently close to $p$, the set $Per(f_v)=\{x\in B_v|f^k_v(x)=x\}$ is compact in $B_v$ - that is, the fixed points of $f^k_v$ all lie away from $\partial B_v$. Consequentially, $f^k_p,f^k_v$ have the same Fixed-Point Index in $B_p,B_v$, respectively - and since we already know the Fixed-Point Index of $f^k_p$ in $B_p$ is non-zero it follows by the discussion above there exists $x_v\in B_v$ s.t. $f^k_v(x_v)=x_v$ - and since for every $i\ne j$, $0\leq i,j<k$, $f^i_v(B_v)\cap f^j_v(B_v)=\emptyset$, the minimal period of $x_v$ is $k$. Moreover, by definition, $x_v$ lies on $T_v$, a periodic trajectory for $F_v$.\\

\begin{figure}[h]
\centering
\begin{overpic}[width=0.5\textwidth]{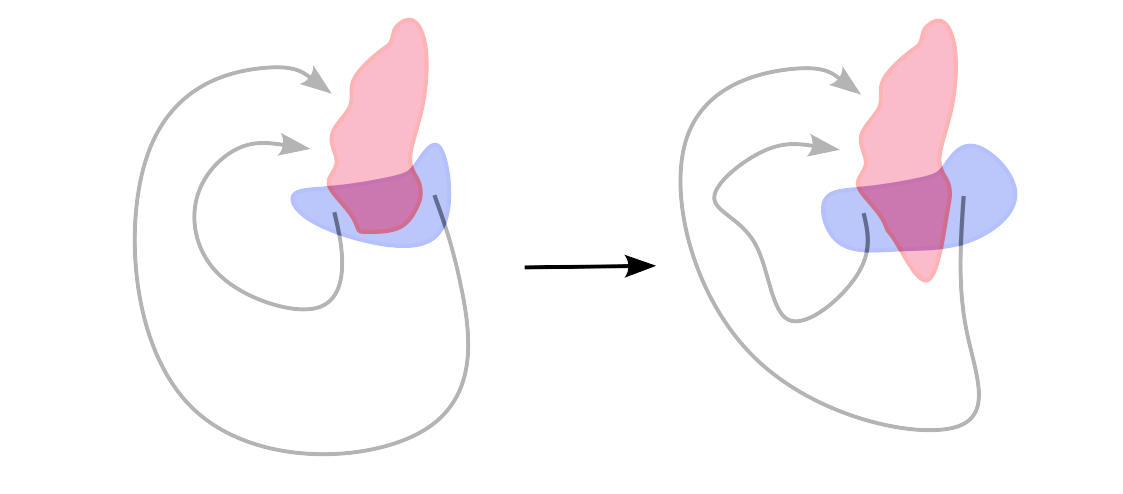}
\put(206,400){$K_p$}
\put(700,400){$K_v$}
\put(170,630){$x$}
\put(10,650){$y$}
\put(915,260){$B_v$}
\put(405,250){$B_p$}
\put(380,350){$f_p(B_p)$}
\put(850,380){$f_v(B_v)$}
\end{overpic}
\caption[Fig31]{The flow lines connecting $B_p,f_p(B_p),...,f^{k}_p(B_p)$ on the left are continuously deformed to the flow lines connecting $B_v,f_v(B_v)...,f^k_v(B_v)$ on the right - i.e., $K_p$ is isotopically deformed to $K_v$. In this illustration, $k=1$.}
\label{pers2}
\end{figure}

We now conclude the proof of Cor.\ref{persistence} by showing the knot type of $T_p,T_v$ is the same. To do so, consider $K_p$, the set of flow-lines connecting $B_p, f_p(B_p),...,f^k_p(B_p)$. It forms a three-dimensional body in $\mathbf{R}^3$, which is knotted with itself (see the illustration in Fig.\ref{pers2}) - and since the periodic trajectory $T_p$ by definition lies inside $K_p$, because the minimal period of $x_p$ is $k$ it follows any crossing on $T_p$ implies a crossing on $K_p$ and vice versa. Additionally, as $K_p$ is composed of the flow lines connecting $B_p, f_p(B_p),...,f^k_p(B_p)$, when we smoothly deform $F_p$ to $F_v$ the body $K_p$ is continuously deformed to $K_v$ - a three-dimensional body generated by the flow lines connecting $B_v,f_v(B_v),...,f^k_v(B_v)$.\\

Since by definition $T_v\subseteq K_v$ and because the minimal period of $x_v$ is (again) $k$, it similarly follows that any crossing on $T_v$ generates a crossing on $K_v$ (see the illustration in Fig.\ref{pers2}). Consequentially, since the three-dimensional bodies $K_v,K_p$ can be deformed to one another by an isotopy of $\mathbf{R}^3$, it follows $T_p,T_v$ must have the same knot type - and by the discussion above, Cor.\ref{persistence} now follows.
\end{proof}

\begin{remark}
    With just a little work, Cor.\ref{persistence} can be generalized to sufficiently small $C^1$ perturbations of trefoil parameters.
\end{remark}
\section{Discussion}
Before we conclude this paper, let us briefly discuss the meaning of Th.\ref{geometric}, Th.\ref{deform} and Th\ref{persistence} - as well as their possible generalizations and the question they raise. As stated at the Introduction, the results of this paper are motivated by the Thurston-Nielsen Classification Theorem - and in more generality, by the theory of topological dynamics on surfaces. Therefore, in the same spirit we are motivated to propose the following generalization of Th.\ref{deform} and Th.\ref{geometric} to non-heteroclinic parameters:

\begin{conj}\label{knot}
Let $p\in P$ be a trefoil parameter for the Rössler system, and recall that for $v\in P$ we denote the corresponding vector field by $F_v$. Then, there exists a neighborhood of $p$ in $P$, $O\subseteq P$, s.t. for every $v\in O$ the following is satisfied:

\begin{itemize}    
\item There exists a smooth vector field $G_v$, hyperbolic on its non-wandering set, s.t. $F_v$ can be smoothly deformed on $\mathbf{R}^3$to $G_v$.
\item Every periodic trajectory $T_1$ of $G_v$ (save perhaps for one) can be deformed by an isotopy of $\mathbf{R}^3$ to a periodic trajectory $T_0$ of $F_v$. In particular, if $T_1,T_2$ are two distinct periodic trajectories for $G_v$, they cannot be deformed to the same $T_0$.
\item The Template generated by $G_v$, $T_v$, is a sub-template of the $L(0,1)$ Template. Consequentially, $F_v$ generates every knot type on $T_v$ (save perhaps for one) as a non-constant periodic trajectory.
\end{itemize}
\end{conj}

Despite its technical formulation, if correct, Conjecture \ref{knot} implies that around trefoil parameters in $P$ the Rössler system is essentially hyperbolic: that is, even though the actual dynamics of the Rössler system may not necessarily be hyperbolic, it is at least qualitatively so - i.e., it can be smoothly deformed to a hyperbolic vector field with a minimal loss of dynamical data along the way. Therefore, if true, Conjecture \ref{knot} implies the Chaotic Hypothesis (see \cite{gal}) holds for the Rössler system - at least on some open region in the parameter space. In particular, Conjecture \ref{knot} implies chaotic dynamics are an essential part of the dynamics generated by the Rössler system - hence these dynamics cannot be easily removed by an arbitrary smooth deformation of the vector field. It is easy to see this would make the vector fields $G_v$ analogues for the Pseudo-Anosov maps known from the Thurston-Nielsen Classification Theorem.\\

Another takeaway from our results (and in particular, from Th.\ref{persistence}) is that it gives a certain framework for the bifurcations of the Rössler system around trefoil parameters. To state this heuristic, let us recall that by Th.\ref{deform}, all the periodic trajectories for $G$ are hyperbolic and isolated - i.e., if $T$ is a periodic trajectory for $G$ and if $\{T_n\}_n$ is a collection of periodic trajectories for $G$ s.t. $T_n\to T$, then the period of $T_n$ wr.t. to $G$ must diverge to $\infty$. Consequentially, the \textbf{Orbit Index} for $T$, $\phi(T)$, (as defined in Sect.2 of \cite{PY2}) is $-1$. Now, let us recall the following fact (see Th.4.2 in \cite{PY2}):
\begin{claim}
\label{contith}    Let $F$ be a $C^3$ vector field on $\mathbf{R}^3$ and let $T$ be an isolated periodic trajectory for $F$ s.t. $\phi(T)\ne0$. Additionally, let $\{F_t\}_t$ be a smooth curve of $C^3$ vector fields, $t>0$, s.t. $F_0=F$. Then, as we smoothly vary $F$, $T$ can be destroyed in one of three ways:
\begin{enumerate}
    \item $T$ closes to a center fixed point by a Hopf bifurcation.
    \item $T$ goes through a Blue Sky catastrophe, i.e., both its period and diameter diverge to $\infty$.
    \item The period of $T$ diverges to $\infty$ while $T$ remains bounded (for example, by undergoing a period-doubling cascade), after which $T$ collapses to an a-periodic trajectory.
    \end{enumerate}
    
\end{claim}

Now, recall every periodic trajectory for $G$ (save possibly for two) is continuously deformed to a periodic trajectory for $F_p$ - and since every periodic trajectory $T$ for $G$ is both isolated and satisfies $\phi(T)=-1$, it follows the periodic trajectories for $G$ all satisfy the assumptions of Th.\ref{contith}. Now, let $C$ denotes the collection of periodic trajectories for $F_p$ which are deformed to some periodic trajectory for $G$ - by Th.\ref{contith} above, we expect the periodic trajectories in $C$ to terminate in one of three ways, as described above.\\

In \cite{MBKPS} it was shown there exist parameters $(a,b,c)$ outside the parameter space $P$ in which the Rössler system goes through a Hopf-bifurcation - leading to the creation of two fixed points (and a periodic trajectory). Even though one may be tempted to assume the $P-$globally continuable trajectories for the Rössler system are eventually all destroyed by a Hopf bifurcation, there in fact exists another possibility, which appears equally probable.\\

\begin{figure}[h]
\centering
\begin{overpic}[width=0.4\textwidth]{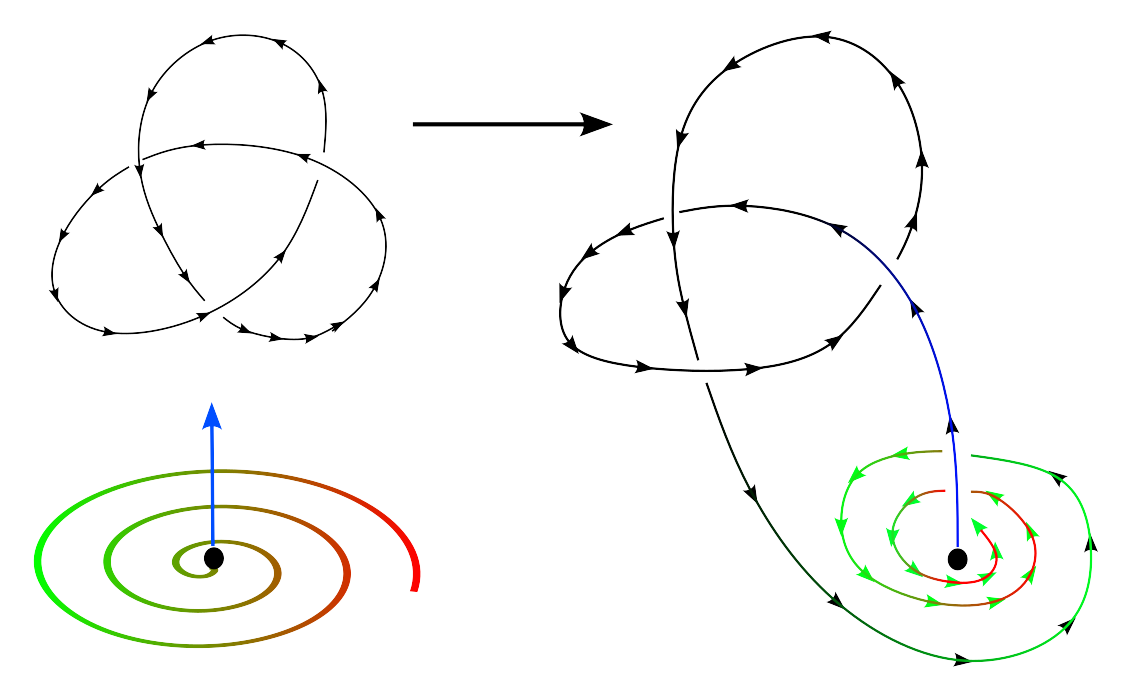}
\end{overpic}
\caption[Fig39]{A periodic trajectory collapsing into a homoclinic trajectory to a saddle focus.}
\label{homoclis}
\end{figure}

To state that possibility, recall a homoclinic trajectory can be generated by a periodic trajectory colliding with a fixed-point, as illustrated in Fig.\ref{homoclis} (see \cite{GKP} for a general review for the bifurcation phenomena associated with such collisions). Therefore, considering the numerically observed connection between the onset of chaos in the Rössler system and homoclinic trajectories (see, for example, \cite{MBKPS}, \cite{BBS} and \cite{G}), it is also equally probable the periodic trajectories in $C$ are destroyed by a collision with a fixed-point (or two) - thus becoming homoclinic or heteroclinic trajectories. As observed in \cite{MBKPS}, there exist many homoclinic trajectories in the parameter space $P$, of varying shapes. As such, inspired by this heuristic, we are motivated to make the following conjecture about the bifurcation structure around trefoil parameters:

\begin{conj}
\label{persss}    Let $p\in P$ be a trefoil parameter, and let $T$ be a periodic trajectory in $C$. Then, there exists a parameter $v\in P$ s.t. as we perturb the Rössler system from $p$ to $v$ through the parameter space, $T$ collapses into a homoclinic or a heteroclinic trajectory for the Rössler system. 
\end{conj}
In a future paper, the author will study the connection between the Orbit Index and the bifurcations of the Rössler and Lorenz systems - in particular, the heuristic described above will be made precise.\\
\begin{figure}[h]
\centering
\begin{overpic}[width=0.4\textwidth]{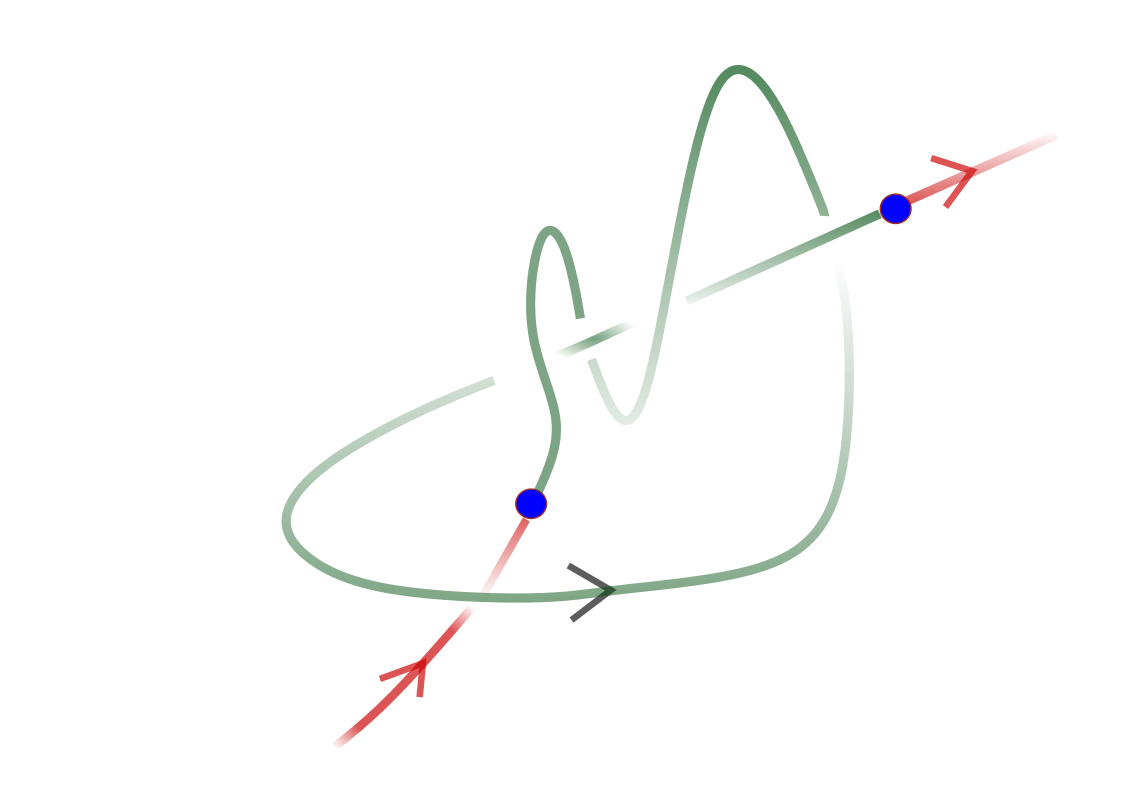}
\put(380,245){$P_{In}$}
\put(780,460){$P_{Out}$}
\end{overpic}
\caption[Fig39]{A heteroclinic knot $\Lambda$ which is more complex than a trefoil.}
\label{fig39}
\end{figure}

As a final remark, let us recall the $L(0,1)$ Template is not a universal template (see \cite{HW}) - that is, it does not encode every knot type in existence. Motivated by this fact, we conjecture Th.\ref{deform}, Th.\ref{geometric} and Th.\ref{persistence} can be generalized to other heteroclinic scenarios in $P$ - provided they generate heteroclinic knots complex at least like the trefoil. For example, consider the heteroclinic knot in Fig.\ref{fig39}. If Th.\ref{geometric} can be generalized to such a heteroclinic knot, the template corresponding to the flow would have to include a horseshoe on \textbf{$3$ }symbols. Therefore, in this spirit, we propose the following conjecture:

\begin{conj}\label{complexity}
   Let $p\in P$ be a parameter which generates a heteroclinic knot, $\Lambda\subseteq S^3$. Then, the more $\Lambda$ is knotted with itself, the more knot types are in the template $T_p$ associated with the Rössler system corresponding to $p$ - and in particular, every knot type on $T_p$ (save possibly for two) is realized as a periodic trajectory for the corresponding Rössler system. 
\end{conj} 
Conjecture \ref{complexity} can be heuristically stated as follows - the more complex the heteroclinic knot $\Lambda$ is, the more complex the dynamics of the corresponding Rössler system are. If true, Conjecture \ref{complexity}provides an analytic explanation to two well-known numerical observations. The first is that as the parameters $(a,b,c)$ are varied, the number of symbols in the first-return map of the Rössler attractor increases - i.e., the symbolic dynamics change from two symbols into three, four, and so on (see, for example, \cite{BBS},\cite{MBKPS},\cite{Le}). The second observation is that the same is true for the associated templates - whose topology also changes with the variation of $(a,b,c)$, thus forcing the creation of more symbols in the first return map of the attractor (see \cite{RO},\cite{Le}).\\

Before concluding this paper, let us remark that both in this paper and in \cite{I} the question of the existence (or non-existence) of the numerically observed Rössler attractor was left unaddressed. In a future paper it will be proven that given a parameter $p\in P$, provided the invariant manifolds of the fixed point $P_{Out}$ satisfy a certain topological condition, there exists an attractor for the flow. Moreover, inspired by Conjecture \ref{knot} and the Betsvina-Handel Algorithm (see \cite{BeH}) we will prove the dynamics around (and beyond) the said attractor can be reduced to those of a one-dimensional model: the Logistic Family of quadratic polynomials.

\printbibliography

@article{Ross76,
        author = "Rössler, O.E.",
        title = "An equation for continuous chaos",
        journal = "Physics Letters A",
        volume = "57",
        year = "1976",
        pages = "397-398",
}

@article{XSYS03,
        author = "Yang, X.S. and Yu Y., and Zhang S.",
        title = " A new proof for existence of horseshoe in the Rössler system",
          journal = "Chaos, Solitons, and Fractals",
        volume = "18",
        year = "2003",
        pages = "223-227",
}

@article{KY,
        author = "Kennedy, J., and Yorke, J.A.",
        title = "Topological Horseshoes",
          journal = "Transactions of the American Mathematical Society",
        volume = "353",
        year = "2001",
        pages = "2513–2530",
}

@article{L,
        author = "Lorenz, E.N.",
        title = "Deterministic nonperiodic flow",
          journal = "Journal of Atmospheric Sciences",
        volume = "20",
        year = "1963",
        pages = "130-141",
}

@article{S,
        author = "Smale, S.",
        title = "Differentiable dynamical systems",
          journal = "Bull. Amer. Math. Soc.",
        volume = "73",
        year = "1967",
        pages = "747-817",
}

@article{LeS,
        author = "Shilnikov, L.",
        title = "A case of the existence of a denumerable set of periodic motions",
          journal = "Sov. Math. Dok.",
        volume = "6",
        year = "1967",
        pages = "163-166",
}

@article{LiLl,
        author = "Lima, M.F.S., and Llibre, J.",
        title = "Global dynamics of the Rössler system with
conserved quantities",
          journal = "J. Phys. A: Math. Theor. ",
        volume = "44",
        year = "2011"
}

@article{MBKPS,
        author = "Malykh, S., and Bakhanova, Y., and Kazakov, A., and Pusuluri, K., and Shilnikov, A.",
        title = "Homoclinic chaos in the Rössler model",
          journal = "Chaos",
        volume = "30",
        year = "2020",
}

@article{BBS,
        author = "Barrio, R., and Blesa, F. and Serrano, S.",
        title = "Topological Changes in Periodicity Hubs of Dissipative Systems",
          journal = "Phys. Rev. Lett.",
        volume = "108, 214102",
        year = "2012",
}

@article{BB,
        author = "Bonatti, C., and Beguin, F.",
        title = "Flots de smale en dimension 3: pr´esentations finies de
voisinages invariants d’ensembles selles",
          journal = "Topology",
        volume = "41(1)",
        year = "2002",
        pages="119-162"
}

@article{BW,
        author = "Birman, J.S., and Williams, R.F.",
        title = "Knotted periodic orbits in dynamical systems II: Knot holders for fibered knots",
          journal = "Contemporary Mathematics",
        volume = "20",
        year = "1983",
        pages="1-60"
}

@article{HW,
        author = "Holmes, P., and Williams, R.F.",
        title = "Knotted periodic orbits in suspensions of Smale's horseshoe: Torus knots and bifurcation sequences",
          journal = "Archive for Rational Mechanics and Analysis",
        volume = "90",
        year = "1985",
        pages="115-194"
}

@article{K,
        author = "Krischenko, A.P.",
        title = "Estimations of domains with cycles",
          journal = "Computers and Mathematics with Applications",
        volume = "34",
        year = "1997",
        pages = "325-332"
}

@article{G,
        author = "Gallas, J.C.",
        title = "The Structure of Infinite Periodic and Chaotic Hub Cascades in Phase Diagrams of Simple Autonomous Flows",
          journal = "International Journal of Bifurcation and Chaos",
        volume = "20(2)",
        year = "2010",
        pages = "197-211"
}

@article{GKP,
        author = "Gaspard, P., and Kapral, R., and Nicolis, G.",
        title = "Bifurcation Phenomena near Homoclinic Systems: 
A Two-Parameter Analysis",
          journal = "Journal of Statistical Physics",
        volume = "35 (5/6)",
        year = "1984",
        pages = "597-727"
}

@article{BB2,
        author = "Barrio, R., and Blesa, F., and Serrano, S.",
        title = "Unbounded dynamics in dissipative flows: Rössler model",
          journal = "Chaos",
        volume = "242",
        year = "2014",
}

@article{RO,
        author = "Rosalie, M.",
        title = "Templates and subtemplates of Rössler attractors from a bifurcation diagram",
          journal = " Journal of Physics A: Mathematical and Theoretical, IOP Publishing",
        volume = "49(31)",
        year = "2016",
}

@article{Le,
        author = "Letellier, C., and Dutertre, P., and Maheu, B.",
        title = "Unstable periodic orbits and templates of the Rössler system: Toward a systematic topological characterization",
          journal = " Chaos",
        volume = "5, 271",
        year = "1995",
}

@article{B,
        author = "Benini, A.M.",
        title = "A survey on MLC, Rigidity, and related topics",
          journal = "arXiv:1709.09869",
         year = "2017",
  }

@book{SSTC,
  author = "Chua, L.O., and Shilnikov, L.P., and Shilnikov, A. and Turaev, D.V.",
  year = "2001",
  title = "Methods of Qualitative Theory in Nonlinear Dynamics, Part II",
  publisher = "New Jersey: World Scientific"
}

@inbook{SR,
    author = "Barrio, R., and Shilnikov, A., and Shilnikov, L.P.",
    title = {Chaos, CNN, Memristors and beyond – a festschrift for Leon Chua},
    publisher = "World Scientific",
    year = "2013",
    chapter = "33 - \textit{Symbolic Dynamics and Spiral Structures due to the Saddle Focus Bifurcations}",
}

@article{Bo,
        author = "Boyland, P.",
        title = "Topological methods in surface dynamics",
          journal = "Topology and its Applications",
        volume = "58 (3)",
        year = "1994"
}

@article{M,
        author = "MacKay, R.S.",
        title = "Complicated dynamics from simple topological hypotheses",
          journal = "Phil. Trans. R. Soc. A.",
        volume = "359",
        year = "2001",
        pages = "1479-1496"
}

@article{BeH,
        author = "Betsvina, M. and Handel, M.",
        title = "Train-tracks for surface homeomorphisms",
          journal = "Topology",
        volume = "34 (1)",
        year = "1995",
        pages = "109-140"
}

@article{Han,
        author = "Handel, M.",
        title = " Global shadowing of pseudo-Anosov homeomorphisms",
          journal = "Ergodic Theory and Dynamical Systems",
        volume = "5 (3)",
        year = "1985",
        pages = "373-377"
}

@online{I,
    author = "Igra, E.",
    title = "Knots and Chaos in the Rössler System",
    url  = "https://arxiv.org/abs/2306.04772v4",
}

@article{Pi,
    author = "Pinsky, T.",
    title ="Analytical study of the Lorenz system: Existence of infinitely many periodic orbits and their topological characterization
" ,
    journal ="Proceedings of the National Academy of Sciences" ,
    volume = "120",
    year ="2023"
}

@book{Dold,
  author = "Dold, A.",
  year = "1972",
  title = "Lectures on Algebraic Topology",
  publisher = "Springer"
}

@article{Lo,
        author = "Lorenz, E.N",
        title = "Deterministic Nonperiodic Flow",
          journal = "Journal of the Atmospheric Sciences",
        volume = "20",
        year = "1963",
        pages = "130-141"
}

@book{KNOTBOOK,
  author = "Ghrist, R.W., and Holmes, P.J., and Sullivan, M.C.",
  year = "1997",
  title = "Knots and Links in Thrre-Dimensional Flows",
  publisher = "Springer"
}

@inbook{PY2,
    author = "Alligood, K.T., and Paret, J.M., and Yorke, J.A.",
    title = {Geometric Dynamics},
    publisher = "Springer Verlag",
    year = "1983",
    chapter = "1 - \textit{An index for the continuation of relatively isolated sets of periodic orbits}",
}

@article{gal,
        author = "Gallavotti, G.",
        title = "Entropy, thermostats, and chaotic hypothesis",
          journal = "Chaos",
        volume = "16",
        year = "2006",
}
\end{document}